%% file: conformalequivalenceoncompactsets.tex
\documentclass[12pt,a4paper,notitlepage]{article}
\include{packages}

\include{usersetcommands}
\include{userinfo}

\begin{document}

\maketitle
\input{abstract}
\input{_section1_overview}
\input{_section2_constructionofPC_a}
\input{_section3_mainproof}
\input{_section4_futurework}

\bibliographystyle{plain}
\bibliography{refs}

\end{document}

%% file: packages.tex
\usepackage{fullpage, amsfonts, amsthm, amsmath, setspace, graphicx, amssymb, float, colonequals, hyperref, float}
\usepackage[usenames,dvipsnames]{xcolor}

%% file: usersetcommands.tex
\newtheorem{theorem}{Theorem}[section]
\newtheorem{lemma}[theorem]{Lemma}

\newtheorem*{q1}{Question 1}
\newtheorem*{q2}{Question 2}

\numberwithin{subcase}{case}
\numberwithin{subsubcase}{subcase}
\numberwithin{claim}{theorem}

\newenvironment{definition}[1][Definition]{\begin{trivlist}
\item[\hskip \labelsep {\bfseries #1}]}{\end{trivlist}}

\newenvironment{example}[1][Example]{\begin{trivlist}
\item[\hskip \labelsep {\bfseries #1}]}{\end{trivlist}}

\newtheorem{innercustomthm}{Theorem}
\newenvironment{customthm}[1]
  {\renewcommand\theinnercustomthm{#1}\innercustomthm}
  {\endinnercustomthm}

\newenvironment{corollary-mannum}[2][Corollary]{\begin{trivlist}
\item[\hskip \labelsep {\bfseries #1}\hskip \labelsep {\bfseries #2}]}{\end{trivlist}}
\newenvironment{conjecture-mannum}[2][Conjecture]{\begin{trivlist}
\item[\hskip \labelsep {\bfseries #1}\hskip \labelsep {\bfseries #2}]}{\end{trivlist}}

\newcommand{\ds}{\displaystyle}

%% file: userinfo.tex
\title{\bf Conformal equivalence of analytic functions on compact sets.}
\author{Trevor Richards\thanks{Email: richardst@wlu.edu}\vspace{6mm}\\{\em Mathematics  Department, Washington and Lee University}\\{\em Lexington, VA United States}}

%% file: abstract.tex
\begin{abstract}
In this paper we present a geometric proof of the following fact.  Let $D$ be a Jordan domain in $\mathbb{C}$, and let $f$ be analytic on $cl(D)$.  Then there is an injective analytic map $\phi:D\to\mathbb{C}$, and a polynomial $p$, such that $f\equiv p\circ\phi$ on $D$ (that is, $f$ has a polynomial conformal model $p$).
\end{abstract}

%% file: _section1_overview.tex
\section{HISTORY AND OVERVIEW}\label{sect: History and overview.}%

It is a well known fact that if $G\subset\mathbb{C}$ is an open set, and $f:G\to\mathbb{C}$ is analytic, and $z\in G$ is a zero of $f$ with multiplicity $k$, then there is a Jordan domain $G_0\subset G$ which contains $z$, and an injective analytic map $\phi$ from $G_0$ onto some disk centered at the origin, such that $f\equiv\phi^k$ on $G_0$.  Of course if $cl(G_0)\subset G$, then $\partial G_0$ is contained in a level curve of $f$ (ie a connected set on which $|f|$ is constant).

In fact we can generalize this as follows.

\begin{theorem}\label{thm: Equiv to pure power.}
If $\lambda$ is any level curve of $f$ in $G$, and $G_0$ is a bounded face of $\lambda$ such that $G_0\subset G$, and $G_0$ contains a single distinct zero $z$ of $f$ with multiplicity $k$, then there is an injective analytic map $\phi$ from $G_0$ onto a disk $D$ centered at zero such that $f\equiv\phi^k$ on $G_0$.
\end{theorem}

To see this fact, one need only pull the restriction of $f$ to $G_0$ back to the unit disk $\mathbb{D}$ via a Riemann map $\phi_{G_0}$ for $G_0$, and observe that the resulting function on $\mathbb{D}$ is a constant multiple of a finite Blaschke product, which has a single zero with multiplicity $k$.  By tweaking $\phi_{G_0}$ in the natural way we can obtain the result of Theorem~\ref{thm: Equiv to pure power.}.

Theorem~\ref{thm: Equiv to pure power.} may be seen as the simplest instance of a problem which is at the core of the recently developing study of the conformal equivalence of meromorphic functions.  There are two basic questions here.

\begin{q1}
Find necessary and sufficient conditions on the function pairs $(f_1,D_1)$ and $(f_2,D_2)$ (where $D_i$ is a domain in $\mathbb{C}$, and $f_i$ is meromorphic on $D_i$) for there to exist a conformal map $\phi:D_1\to D_2$ satisfying $f_1\equiv f_2\circ\phi$ on $D_1$.
\end{q1}

In this case we say that the functions $f_1$ and $f_2$ are \textit{conformally equivalent}.  In the special case that each domain $D_i$ is a Jordan domain, and each $f_i$ satisfies 1)${f_i}'\neq0$ on $\partial D_i$ and 2) $|f_i|=1$ on $\partial D_i$, the author has given a solution in~\cite{Ri} to this problem in terms of the configurations of the critical level curves of $f_1$ and $f_2$ in $D_1$ and $D_2$.

The second and better trod basic question concerns the existence of a rational (or polynomial) conformal model to a function on a domain.

\begin{q2}[Conformal modeling question]
Let $D$ be a domain in $\mathbb{C}$, and let $f$ be meromorphic on $D$.  Find necessary and sufficient conditions for there to exist a rational function (or polynomial) $r$ and an injective analytic map $\phi:D\to\mathbb{C}$ satisfying $f\equiv r\circ\phi$ on $D$.
\end{q2}

In this case we would say that $r$ is a \textit{rational conformal model} for $f$ on $D$.  We will now outline the known partial answers to the conformal modeling question and discuss the various methods which different authors have taken to approach the question.

The first generalization of Theorem~\ref{thm: Equiv to pure power.} comes by removing the assumption $f$ contains a single distinct zero in $G_0$ as in the following Theorem~\ref{thm: GfBp equiv to a polynomial.}.  In this case the conclusion ``$f\equiv\phi^k$'' is replaced by the more general functional equation ``$f\equiv p\circ\phi$ for some degree $k$ polynomial $p\in\mathbb{C}[z]$''.

\begin{theorem}\label{thm: GfBp equiv to a polynomial.}
If $\lambda$ is any level curve of $f$ in $G$, and $G_0$ is some bounded face of $\lambda$ such that $G_0\subset G$, and $k$ is the number of zeros of $f$ in $G_0$ counting multiplicity, then there is an injective analytic map $\phi$ from $G_0$ onto a Jordan domain $D\subset\mathbb{C}$, and a degree $k$ polynomial $p\in\mathbb{Z}$ such that $f\equiv p\circ\phi$ on $G_0$.
\end{theorem}

To see this, we again pull $f$ back to the unit disk via a Riemann map for $G_0$.  The resulting function on $\mathbb{D}$ is again a constant multiple of a finite Blaschke product (this time with possibly several distinct zeros).  The desired result now follows from the following Theorem~\ref{thm: Finite Blaschke product-polynomial conformal equivalence.}.

\begin{customthm}{\ref{thm: Finite Blaschke product-polynomial conformal equivalence.}}
Given any finite Blaschke product $B$, there is an injective analytic map $\phi:\mathbb{D}\to\mathbb{C}$, and a polynomial $p\in\mathbb{C}[z]$  such that $B\equiv p\circ\phi$ on $\mathbb{D}$.
\end{customthm}

Theorem~\ref{thm: Finite Blaschke product-polynomial conformal equivalence.} has at least four fundamentally different proofs~\cite{EKS,LS,Ri,Yo}, and at the end of this section we will briefly discuss each proof.

In 2013 on the internet mathematics forum {\em math.stackexchange.com}, the author posted the following generalization of Theorem~\ref{thm: Finite Blaschke product-polynomial conformal equivalence.} as a conjecture, and received a proof~\cite{LS} (though not yet published in a peer reviewed journal) from users George~Lowther and David~Speyer.  This conjecture (which in light of the proof posted on \textit{math.stackexchange.com} we now refer to as a theorem) may be reformulated as follows.

\begin{theorem}\label{thm: Conformal equivalence, most general.}
If $D\subset\mathbb{C}$ is open and bounded, and $f$ is meromorphic on $cl(D)$, then there is an injective analytic map $\phi:D\to\mathbb{C}$, and a rational function $r\in\mathbb{C}(z)$ such that $f\equiv r\circ\phi$ on $D$.
\end{theorem}

Our main goal in this paper is to employ the geometric level curve methods used in~\cite{Ri} to prove the following analytic and simply connected case of Theorem~\ref{thm: Conformal equivalence, most general.}.

\begin{customthm}{\ref{thm: Analytic implies equiv to poly.}}
If $D\subset\mathbb{C}$ is a Jordan domain, and $f$ is analytic on $cl(D)$, then there is polynomial $p\in\mathbb{C}[z]$ and an injective analytic map $\phi:D\to\mathbb{C}$ such that $f\equiv p\circ\phi$ on $D$.
\end{customthm}

We will prepare for the proof of Theorem~\ref{thm: Analytic implies equiv to poly.} in Section~\ref{sect: Construction of PC_a.} by constructing a set $PC_a$ which parameterizes the possible critical level curve configurations of a complex polynomial.  This set was introduced in~\cite{Ri}, where the following theorem was proved.

\begin{customthm}{\ref{thm: Polynomial critical level curve config existence.}}
Given any critical level curve configuration in $PC_a$, there is a polynomial $p\in\mathbb{C}[z]$ which has this configuration.
\end{customthm}

In Section~\ref{sect: Proof of main result.}, we bootstrap Theorem~\ref{thm: Polynomial critical level curve config existence.} to a proof of Theorem~\ref{thm: Analytic implies equiv to poly.}.  Finally, in Section~\ref{sect: Future work.}, we identify some directions for future work in the area.

\subsection{\small The proof of Ebenfelt, Khavinson, and Shapiro}

In their paper, Ebenfelt et. al. do not discuss the notion of conformal equivalence directly, but rather prove a theorem for fingerprints of smooth curves which may be reinterpreted as the statement of Theorem~\ref{thm: Finite Blaschke product-polynomial conformal equivalence.}.

For a smooth shape $\Gamma$, we let $\phi_-$ and $\phi_+$ denote normalized Riemann maps from $\mathbb{D}$ and $\widehat{\mathbb{C}}\setminus\overline{\mathbb{D}}$ to the bounded and unbounded faces of $\Gamma$ respectively.  Then we define the fingerprint of $\Gamma$ to be the diffeomorphism $k_\Gamma={\phi_+}^{-1}\circ\phi_-:\mathbb{T}\to\mathbb{T}$.  If $\Gamma'=T(\Gamma)$ where $T(z)=az+b$ for some $a>0$ and $b\in\mathbb{C}$, then $k_{\Gamma'}=k_\Gamma\circ\psi$ for some conformal map $\psi:\mathbb{D}\to\mathbb{D}$ (ie for some degree~$1$ finite Blaschke product $\psi$).

It is known (as a result of some work of Pfluger~\cite{P} in the area of quasiconformal mappings) that the map $\mathcal{F}:\Gamma\mapsto k_\Gamma$ from the collection of smooth shapes $\Gamma$ (modulo translation and scaling) to the diffeomorphisms of $\mathbb{T}$ (modulo precomposition with a conformal self-map of the disk) is in fact a bijection.  Ebenfelt et. al. show that if one restricts $\mathcal{F}$ to the collection of smooth shapes $\Gamma$ which arise as proper polynomial lemniscates (ie lemniscates of a polynomial $p$ such that all zeros of $p$ are contained in the bounded face of $\Gamma$), then the restricted map $\mathcal{F}$ is a bijection onto the collection of diffeomorphisms $S:\mathbb{T}\to\mathbb{T}$ of the form $S=B^{1/n}$ for some degree $n$ Blaschke product (where $n$ is any positive integer).

Observing that, if $\Gamma$ is a proper lemniscate for a degree $n$ polynomial $p$, then the map ${\phi_+}^{-1}=p^{1/n}$, the surjectivity claim for the function $\mathcal{F}$ restricted to proper polynomial lemniscates in the preceding paragraph may easily be seen to be equivalent to our Theorem~\ref{thm: Finite Blaschke product-polynomial conformal equivalence.}.  Their method of proof, then, consists of parameterizing the collection of proper degree $n$ polynomial lemniscates by a certain subset of $\mathbb{C}^{n-2}$, and the collection of degree~$n$ finite Blaschke products by another subset of $\mathbb{C}^{n-1}$.  Viewing now the function $\mathcal{F}$ as mapping the first set to the second, they use Koebe's continuity method to show that the range of $\mathcal{F}$ is a clopen subset of the connected codomain, and thus equal to the entire thing, establishing the surjectivity result.

\subsection{\small The approach of Younsi}

Younsi simplified the picture substantially by bringing the machinery of conformal welding to bear on the problem.  Given a degree~$n$ finite Blaschke product $B$, one conformally welds a copy of the function $z\mapsto z^n$ on the region $\widehat{\mathbb{C}}\setminus\mathbb{D}$
to the function $B$ restricted to $\mathbb{D}$.  The resulting function is conformally equivalent to a proper polynomial, which is conformally equivalent to $B$ on the bounded face of its lemniscate.

This approach has application to a wider setting than just finite Blaschke products.  As with Ebenfelt et. al., Younsi approaches the problem from the standpoint of fingerprints of shapes, and uses it to characterize the fingerprints of proper rational lemniscates (proper in the sense that all zeros of the rational function are in the bounded face, and all poles are in the unbounded face).  Moreover, an identical welding argument as described above can be used to show that the conclusion of Theorem~\ref{thm: Finite Blaschke product-polynomial conformal equivalence.} holds when the finite Blaschke product $B$ is replaced by a ratio of finite Blaschke products $B/A$ (subject to the constraint that $(B/A)'\neq0$ on $\mathbb{T}$) and ``polynomial" is replaced by ``rational function".

Recently Younsi~\cite{RY} (working jointly with the author) has applied these conformal welding techniques to provide a positive answer to the conformal modeling question for meromorphic functions on psuedo-tracts.  That is, it has been shown that if $D$ is a Jordan domain with smooth boundary, and $f$ is meromorphic on $cl(D)$, and $f(\partial D)$ is a Jordan curve, then $f$ may be conformally modeled by a rational function $r$ on $D$.  One of the major advantages of this approach is that in the case described above the rational function $r$ may be taken to have the smallest degree possible (namely the maximal number of preimages under $f$ of any point, counted with multiplicity).

\subsection{\small The approach of Lowther and Speyer}

Users George~Lowther and David~Speyer took a classical complex analysis aproach to the proof of Theorem~\ref{thm: Conformal equivalence, most general.}.  While the proof given on \textit{math.stackexchange.com} applies only to analytic functions on the disk, it readily extends to prove the full strength of Theorem~\ref{thm: Conformal equivalence, most general.}, and this extension is what I will describe here.  One first uses a generalization of Runge's theorem to find a sequence of rational functions $\{r_n\}$ which converge uniformly to the function $f$ on the closure of the domain $cl(D)$, and such that the rational functions match the derivative data of $f$ at the critical points of $f$.  From there, one uses the the open mapping theorem to show that for sufficiently large $n$, there is an injective analytic map $\phi_n:D\to\mathbb{C}$ such that $f=r_n\circ\phi_n$.

This approach has the immediate advantage of answering in the affirmative the conformal modeling question for the largest class of functions of any of the approaches discussed here, namely functions meromorphic on a compact set.  The only apparent downside to this approach is that it does not appear to have any hope of giving information about the degree of the conformal model attained, as the degrees of the rational functions found using Runge's theorem approach infinity.

\subsection{\small The approach of the author}

The author has taken a very different and geometric approach to the subject.  In~\cite{Ri}, we constructed a set $PC$ which parameterized the possible level curve configurations (a purely geometric concept) of a ratio of finite Blaschke products (again not having a critical point on $\mathbb{T}$).  (This construction will be repeated in Section~\ref{sect: Construction of PC_a.}.)  We then showed that any two such ratios whose critical level curve configurations are represented by the same memberm of $PC$ are conformally equivalent.  Finally, we showed that for every possible configuration of critical level curves, there is some polynomial having this configuration, which gives the desired result.

Our approach to Theorem~\ref{thm: Finite Blaschke product-polynomial conformal equivalence.} has two main advantages.  First, as we will see in this paper, it has the advantage of extendability.  While the approach of Younsi does end up giving an answer to the conformal mapping question for a class of functions which our approach at present does not handle (namely, ones with poles) it does not appear that the approach either of Ebenfelt et. al. or of Younsi is likely to have any application to the setting under consideration in this paper in which the only boundary condition imposed on the function $f$ is analyticity (ie. $f$ must be analytic across the boundary of the domain $D$).

Second, our approach demonstrates the rigidity of analytic functions.  The critical level curve configuration of a finite Blaschke product $B$ acts as a sort of skeleton for $B$.  This geometrical object is a conformal invariant of $B$, and the conformal equivalence result referred to above says that, for a given geometric configuration of critical level curves for $B$, there is only one way for the rest of $B$ to fill in analytically around this configuration.  This dependence of the analytic on the geometric is not seen in any of the approaches mentioned above.

Finally, in the proof of Theorem~\ref{thm: Finite Blaschke product-polynomial conformal equivalence.} in~\cite{Ri}, the degree of the polynomial conformal model $p$ of the finite Blaschke product $B$ may be taken to be equal to the degree of $B$.  As we will discuss in Section~\ref{sect: Future work.}, the level curve approach also seems to have some hope of giving a similar result in the setting of Theorem~\ref{thm: Analytic implies equiv to poly.} at some point in the future.

%% file: _section2_constructionofPC_a.tex
\section{CONSTRUCTION OF $PC_a$.}\label{sect: Construction of PC_a.}%

In order to introduce our proof of Theorem~\ref{thm: Analytic implies equiv to poly.}, we will first define the class of functions which were studied in~\cite{Ri}, namely the generalized finite Blaschke products.

\begin{definition}
For $G\subset\mathbb{C}$ and $f:G\to\mathbb{C}$, we say that the pair $(f,G)$ is a \textit{generalized finite Blaschke product} if the following requirements hold. 

\begin{itemize}
\item
$G$ is a Jordan domain.
\item
$f$ may be extended to an analytic function on $cl(G)$.
\item
$|f|=1$ on $\partial G$.
\item
$f'\neq0$ on $\partial G$.
\end{itemize}
\end{definition}

In fact it is not hard to show that such a pair $(f,G)$ is a generalized finite Blaschke product if and only if the composition $f\circ\phi$ is a finite Blaschke product, where $\phi:\mathbb{D}\to G$ is any Riemann map for $G$ (since any ramified covering from $\mathbb{D}$ to $\mathbb{D}$ is a finite Blaschke product).

\begin{definition}
Let ${H_a}'$ denote the set of all generalized finite Blaschke products.  For two members $(f_1,G_1),(f_2,G_2)\in{H_a}'$, we say that $(f_1,G_1)$ and $(f_2,G_2)$ are \textit{conformally equivalent}, and write $(f_1,G_1)\sim(f_2,G_2)$, if there is some bijective analytic map $\phi:G_1\to G_2$ such that $f_1\equiv f_2\circ\phi$ on $G_1$.
\end{definition}

It is easy to see that $\sim$ is an equivalence relation on ${H_a}'$, so we make the following definition.

\begin{definition}
Define $H_a$ to be the set of \textit{$\sim$-equivalence classes} of ${H_a}'$.
\end{definition}

Our proof of Theorem~\ref{thm: Analytic implies equiv to poly.} uses critically Theorem~\ref{thm: Polynomial critical level curve config existence.} that, given any possible critical level curve configuration of a generalized finite Blaschke product, there is a polynomial whose critical level curves (ie the level curves which contain critical points) are in that configuration.  In order to make rigorous this notion of a ``possible critical level curve configuration'' for a generalized finite Blaschke product, we will repeat a construction introduced in~\cite{Ri} of a set which parameterizes these configurations, and which we will call $PC_a$.  (In the original construction of this set, a larger set $PC$ was constructed which also contained the possible critical level curve configurations of more general meromorphic functions.  The $a$-subscript used throughout denotes that the functions whose critical level curve configurations we are capturing with the members of $PC_a$ are analytic on their respective domains.)  We begin by defining the basic building blocks of a member of $PC_a$, namely analytic level curve type sets.

\begin{definition}
A set $\Gamma\subset\mathbb{C}$ is said to be of \textit{analytic level curve type} if it is either a single point, a simple closed path, or a connected finite graph which satisfies the following requirements.

\begin{itemize}
\item
There are evenly many, and more than two, edges of $\Gamma$ incident to each vertex of $\Gamma$ (where we count an edge twice if both of its end points are at a given vertex).

\item
Each edge of $\Gamma$ is incident to the unbounded face of $\Gamma$.
\end{itemize}

Moreover, an analytic level curve type set which has vertices (ie is not a single point or a simple closed path) is said to be of \textit{analytic critical level curve type}.
\end{definition}

It is easy to see (using, for example, the maximum modulus theorem and the open mapping theorem) that any level curve of a generalized finite Blaschke product has the properties which define an analytic level curve type set.  Throughout this paper we will view two analytic level curve type sets $\Gamma_1$ and $\Gamma_2$ as equivalent if there is an orientation preserving homeomorphism $\phi:\mathbb{C}\to\mathbb{C}$ which maps $\Gamma_1$ to $\Gamma_2$.

We now define a set $P_a$ whose members will represent the individual zeros and critical level curves of a generalized finite Blaschke product.  As we describe the features and auxiliary data ascribed to a member of $P_a$ (and later of $PC_a$) we will parenthetically remark on the data for a level curve of a given generalized finite Blaschke product $(f,G)$ which those features and auxiliary data are meant to represent.

There are two types of members of $P_a$, namely those meant to represent zeros of $(f,G)$ (which we will call ``single point members of $P_a$'') and those meant to represent critical level curves of $(f,G)$ (which we will call ``graph members of $P_a$'').  We will begin by describing the single point members of $P_a$.

A single point member $\langle w\rangle_{P_a}$ of $P_a$ consists of the graph consisting of a single vertex $w$ with no edges, to which we add the following pieces of auxiliary data.

\begin{itemize}
\item
We define $H(\langle{w}\rangle_{P_a})\colonequals0$.  (This represents the value $|f|$ takes on $w$.)

\item
We define $Z(\langle{w}\rangle_{P_a})\colonequals k$ for some $k\in\{1,2,\ldots\}$.  (This represents the multiplicity of $w$ as a zero of $f$.)

\end{itemize}

The resulting object, with the above auxiliary data, we denote $\langle{w}\rangle_{P_a}$.

A graph member $\langle\lambda\rangle_{P_a}$ of $P_a$ consists of an analytic critical level curve type graph $\lambda$, to which we add the following pieces of auxiliary data.

\begin{itemize}
\item
We define $H(\langle\lambda\rangle_{P_a})=\epsilon$ for some value $\epsilon\in(0,\infty)$.  (This represents the value $|f|$ takes on $\lambda$.)

\item
For each bounded face $D$ of $\lambda$, we choose an integer $z(D)\in\{1,2,\ldots\}$.  (This represents the number of zeros of $f$ in $D$.)  If $D_1,D_2,\ldots,D_k$ denote the bounded faces of $\lambda$, we define $Z(\langle\lambda\rangle_{P_a})\colonequals\displaystyle\sum_{i=1}^kz(D_i)$.

\item
We distinguish a finite number of points in $\lambda$ in such a manner that for each bounded face $D$ of $\lambda$, there are $z(D)$ distinct distinguished points in $\partial{D}$.  (This represents the points in $\lambda$ at which $f$ takes non-negative real values).

\item
For each vertex $x$ of $\lambda$, we designate a value $a(x)\in[0,2\pi)$.  (This represents the argument of $f$ at $x$.)  We require that this assignment obeys the following rules.
	\begin{itemize}
	\item
$a(x)=0$ if and only if $x$ is a distinguished point of $\lambda$.

	\item
If $D$ is a bounded face of $\lambda$, and $x_1$ and $x_2$ are distinct vertices of $\lambda$ in $\partial{D}$ such that $a(x_1)\geq{a(x_2)}$, then there is some distinguished point $z\in\partial{D}$ such that $x_1,z,x_2$ is written in increasing order as they appear with respect to positive orientation around $\partial{D}$.  (This reflects the fact that since $f$ is analytic, $\arg(f)$ is increasing as $\partial{D}$ is traversed with positive orientation.)

\end{itemize}
\end{itemize}

The resulting object, with the above auxiliary data, we denote $\langle\lambda\rangle_{P_a}$, and we define $P_a$ to be the set of all such $\langle\lambda\rangle_{P_a}$ and $\langle{w}\rangle_{P_a}$.

Throughout this paper, $\langle{w}\rangle_{P_a}$ will be used to refer to single point members of $P_a$, $\langle\lambda\rangle_{P_a}$ will be used for graph members of $P_a$, and $\langle\xi\rangle_{P_a}$ will be used when we do not wish to distinguish between the two types of members of $P_a$.

Each member of $PC_a$ consists of a collection of members of $P_a$ arranged in different ways with respect to each other.  There are two aspects of this.  First, which graphs are in which bounded faces of which other graphs, and second, the rotational orientation of each graph with respect to the others.

We do this recursively.  To initialize our recursive construction, a level $0$ member of $PC_a$ will be a single point member $\langle w\rangle_{P_a}$ of $P_a$ viewed as a member of $PC_a$, with no additional data (now written $\langle w\rangle_{PC_a}$).

Let $n>0$ be given.  Choose some graph member $\langle\lambda\rangle_{P_a}$ of $P_a$.  We will now construct a level $n$ member $\langle\lambda\rangle_{PC_a}$ of $PC_a$ as follows.  We have two steps.

\begin{enumerate}
\item
For each bounded face $D$ of $\lambda$, we choose some level $k<n$ member $\langle\xi_D\rangle_{PC_a}$ of $PC_a$ and assign it to $D$.  These assignments must satisfy the following restrictions.

\begin{itemize}
\item
$Z(\langle\xi_D\rangle_{P_a})=z(D)$.  (This represents the fact that if $\lambda$ is a level curve of $f$ in $G$, then either $f$ has a single distinct zero in $D$, or all zeros of $f$ in $D$ are contained in the bounded faces of some single critical level curve $\xi_D$ of $f$ in $D$.  This fact was proved by the author in~\cite{Ri}, and will be discussed further at the end of this section.)

\item
$H(\langle\xi_D\rangle_{P_a})<H(\langle\lambda\rangle_{P_a})$.  (This follows for level curves of analytic functions in view of the maximum modulus theorem.)

\item
At least one of the $\langle\xi_D\rangle_{PC_a}$'s is a level $n-1$ member of $PC_a$  (This is to ensure that $\langle\lambda\rangle_{PC_a}$ was not constructed at any earlier recursive level.)
\end{itemize}

This determines recursively which graphs may lie in which bounded faces of which other graphs.

\item
For each bounded face $D$ of $\lambda$, we choose a map $g_D$ (which we will call a ``gradient map'') from the distinguished points of $\lambda$ in $\partial D$ to the distinguished points in $\xi_D$.  (In the context of level curves of an analytic function $f$, $g_D(y)=x$ means that $y$ and $x$ are connected by a gradient line of $f$.)  This map must satisfy the following restriction.

\begin{itemize}
\item
$g_D$ must preserve the orientation of the distinguished points.  That is, if $y^{(1)},\ldots,y^{(z(D))}$ are the distinguished points of $\lambda$ in $\partial D$ listed in order of their appearance about $\partial D$, then the order in which the critical points in $\xi_D$ appear around $\xi_D$ is exactly $g_D(y^{(1)}),\ldots,g_D(y^{(z(D))})$.  (This represents the fact that if $\lambda$ is a level curve of a function $f$ with bounded face $D$, and $\xi_D$ is the critical level curve of $f$ in $D$ which contains all the zeros of $f$ in its bounded faces, and $\widehat{D}$ denotes the portion of $D$ which is exterior to $\xi_D$, then the gradient lines of $f$ in $\widehat{D}$ cannot cross, since $\widehat{D}$ contains no critical points of $f$.)
\end{itemize}
\end{enumerate}

We let $\langle\lambda\rangle_{PC_a}$ denote the resulting object.  The collection of all such level $0$ objects $\langle w\rangle_{PC_a}$, and level $n>0$ objects $\langle\lambda\rangle_{PC_a}$, we call $PC_a$, and we call $PC_a$ the set of possible level curve configurations.

We adopt the same convention of $w$, $\lambda$ or $\xi$ for members of $PC_a$ as we did for members of $P_a$, namely that level $0$ members of $PC_a$ we denote by $\langle{w}\rangle_{PC_a}$, level $n>0$ members of $PC_a$ we denote by $\langle\lambda\rangle_{PC_a}$, and if we do not wish to specify the level of a member of $PC_a$ we will denote it by $\langle\xi\rangle_{PC_a}$.

\begin{example}
Following is a visual example of how a member of $PC_a$ is constructed.  We begin with a graph member $\langle\lambda\rangle_{P_a}$ of $P_a$ (here with vertex $z$ and bounded faces $D_1$ and $D_2$) which has auxiliary data such as $H(\langle\lambda\rangle_{P_a})=\frac{1}{2}$, $a(z)=\frac{\pi}{4}$, $z(D_1)=4$ and $z(D_2)=2$, and the marked distinguished points in $\lambda$.

\begin{figure}[H]
\centering
	\includegraphics[width=0.75\textwidth]{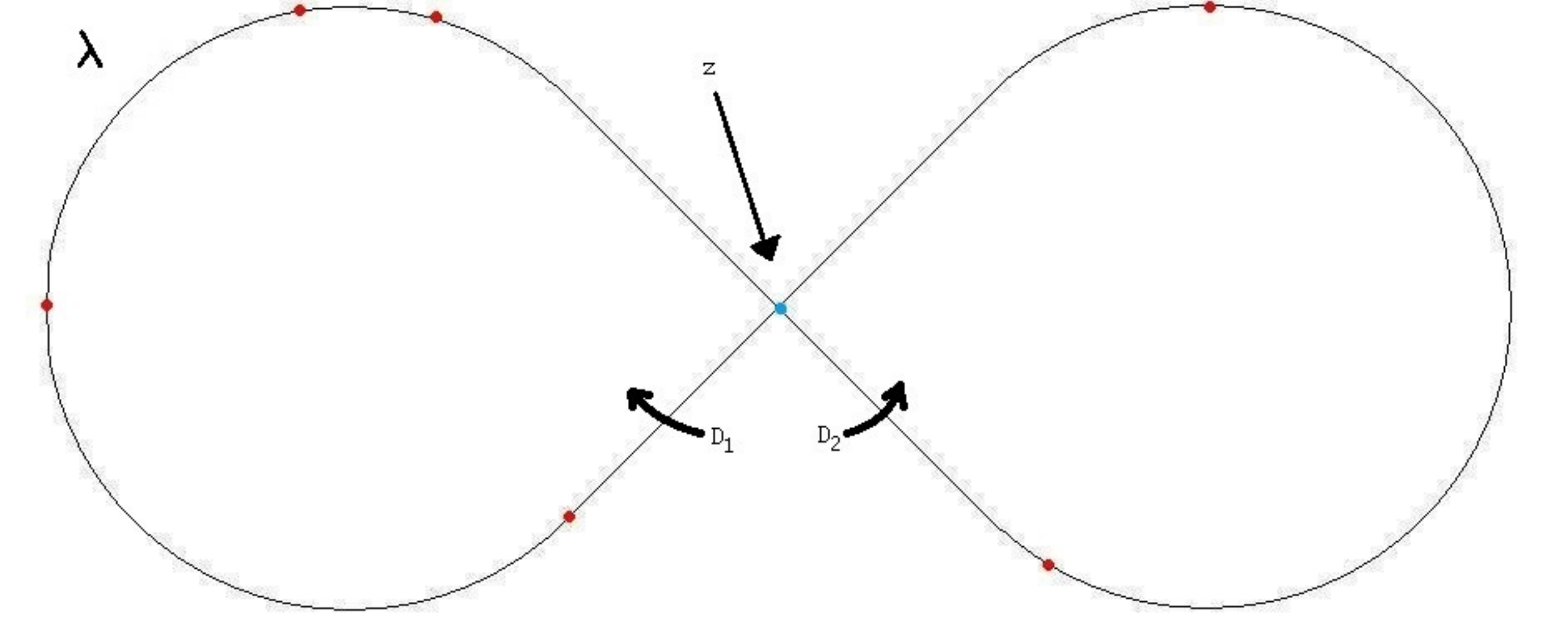}
           \caption{Member of $P_a$}
	\label{fig:MemberofPCStep1}

\end{figure}

We then assign a member of $PC_a$ to each face of $\lambda$.  (Dashed lines represent the action of the gradient maps.)

\begin{figure}[H]
\centering
	\includegraphics[width=0.75\textwidth]{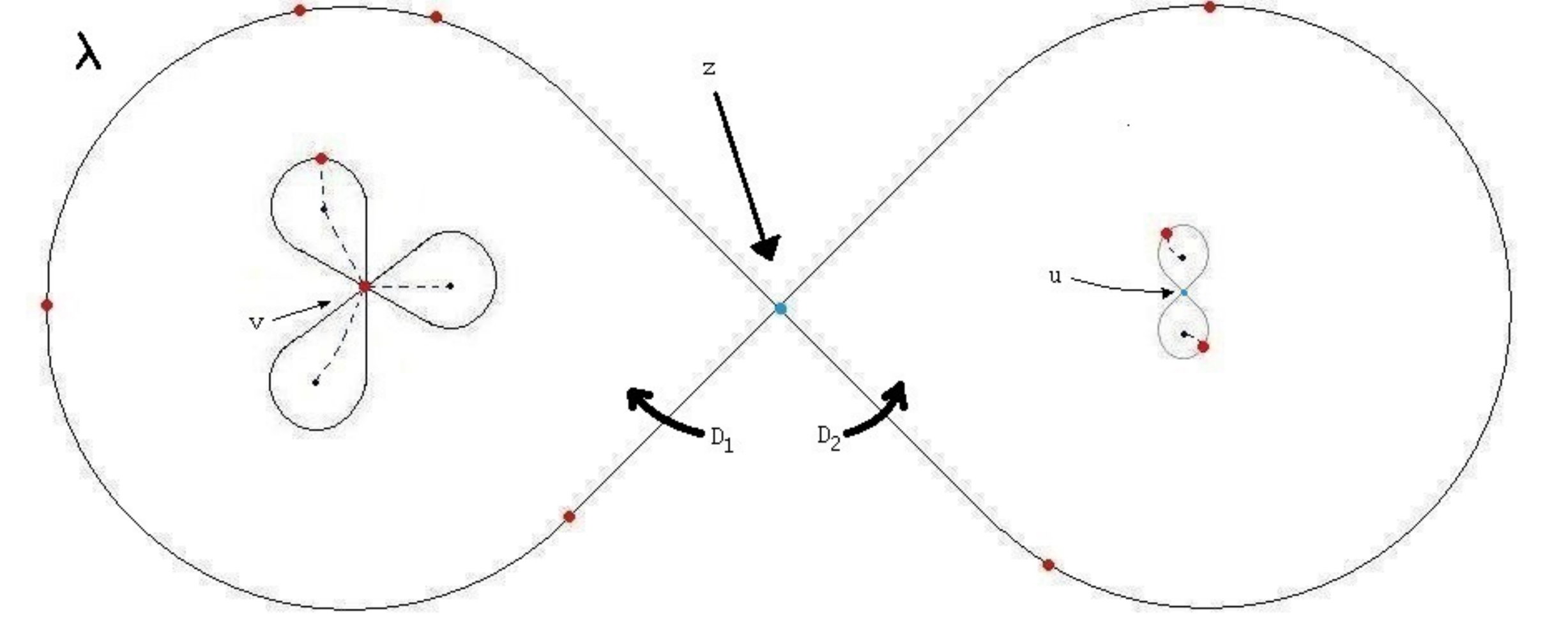}
           \caption{Assignment of members of $PC_a$ to the $D_1$ and $D_2$}
	\label{fig:MemberofPCStep2}

\end{figure}

Finally, we designate a gradient map from the distinguished points in $\lambda$ to the distinguished points in the assigned members of $PC_a$.

\begin{figure}[H]
\centering
	\includegraphics[width=0.75\textwidth]{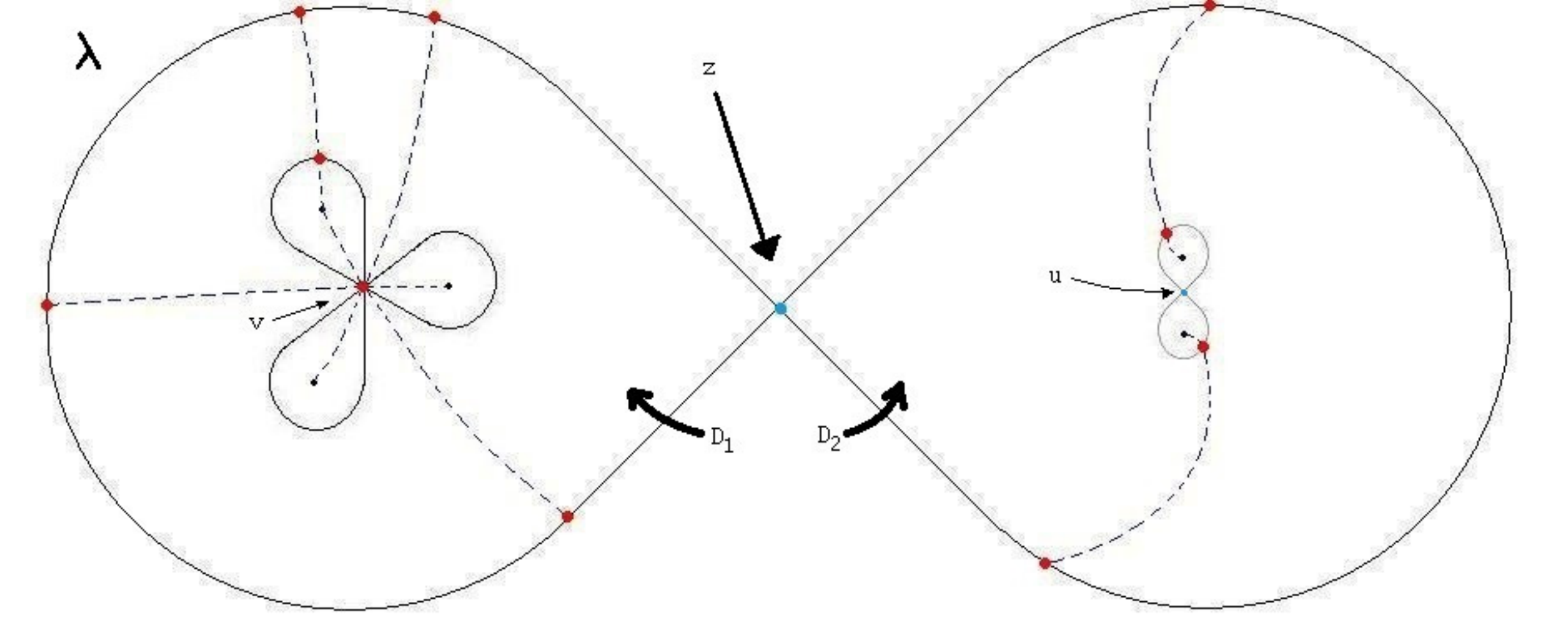}
           \caption{Designation of Gradient Maps}
	\label{fig:MemberofPCStep3}

\end{figure}

The resulting object we denote $\langle\lambda\rangle_{PC_a}$.

\end{example}

In~\cite{Ri}, the author proved the following Theorem~\ref{thm: Level curves separated by critical level curve.}.

\begin{theorem}\label{thm: Level curves separated by critical level curve.}
If $(f,G)$ is a generalized finite Blaschke product, and $L_1$ and $L_2$ are level curves of $f$ in $G$, each of which is in the unbounded face of the other, then there is a critical point $w\in G$ such that $L_1$ and $L_2$ are contained in different bounded faces of the level curve of $f$ containing $w$.
\end{theorem}

Theorem~\ref{thm: Level curves separated by critical level curve.} implies that if $\lambda$ is a level curve of a generalized finite Blaschke product $(f,G)$, and $D$ is a bounded face of $\lambda$, then either $D$ contains a single distinct zero $z$ of $f$, or there is some unique critical level curve $\lambda_D$ of $f$ in $D$ such that each zero and critical point of $f$ in $D$ is either in $\lambda_D$, or in one of the bounded faces of $\lambda_D$ (in the latter case we call $\lambda_D$ the maximal critical level curve of $f$ in $D$).  If $D$ contains a single distinct zero of $f$, then each level curve of $f$ in $D$ consists of simple smooth closed curve containing $z$ in its bounded face.  If $f$ has a maximal critical level curve $\lambda_D$ in $D$, then each level curve of $f$ in $D$ which is exterior to $\lambda_D$ consists of a simple smooth closed curve containing $\lambda_D$ in its bounded face.  In view of these facts which are implied by Theorem~\ref{thm: Level curves separated by critical level curve.}, we conclude that the critical level curves of a generalized finite Blaschke product $(f,G)$ form a member of $PC_a$, which we will denote by $\Pi(f,G)$.

It was further shown in~\cite{Ri} that the map $\Pi$ respects conformal equivalence of generalized finite Blaschke products, as described below in Theorem~\ref{thm: Conformal equivalence same as topological equivalence.}.

\begin{theorem}\label{thm: Conformal equivalence same as topological equivalence.}
For any generalized finite Blaschke products $(f_1,G_1)$ and $(f_2,G_2)$, $(f_1,G_1)\sim(f_2,G_2)$ if and only if $\Pi(f_1,G_1)=\Pi(f_2,G_2)$.
\end{theorem}

Note that Theorem~\ref{thm: Conformal equivalence same as topological equivalence.} shows that, on the one hand, $\Pi$ viewed as acting on the $\sim$-equivalence classes of generalized finite Blaschke products in $H_a$ is well defined and, on the other hand, $\Pi:H_a\to PC_a$ is injective.  It was also shown in~\cite{Ri} that each member of $PC_a$ represents the critical level curve configuration for some polynomial as follows.

\begin{theorem}\label{thm: Polynomial critical level curve config existence.}
Given any critical level curve configuration $\langle\xi\rangle_{PC_a}\in PC_a$, there is a polynomial in $\mathbb{C}[z]$ whose critical level curves form the configuration $\langle\xi\rangle_{PC_a}$.
\end{theorem}

Combining Theorem~\ref{thm: Conformal equivalence same as topological equivalence.} with Theorem~\ref{thm: Polynomial critical level curve config existence.}, we obtain the fact that each generalized finite Blaschke product is conformally equivalent to a polynomial.

\begin{theorem}\label{thm: Finite Blaschke product-polynomial conformal equivalence.}
Given any generalized finite Blaschke product $(f,G)$, there is a polynomial $p\in\mathbb{C}[z]$ and an injective analytic map $\phi:G\to\mathbb{C}$ such that $f\equiv p\circ\phi$ on $G$.
\end{theorem}

In the next section we will give the proof of Theorem~\ref{thm: Analytic implies equiv to poly.}.

%% file: _section3_mainproof.tex
\section{PROOF OF THEOREM \ref{thm: Analytic implies equiv to poly.}}\label{sect: Proof of main result.}

We now proceed to the proof of our main result.

\begin{theorem}\label{thm: Analytic implies equiv to poly.}
If $D\subset\mathbb{C}$ is open and bounded, and $f$ is meromorphic on $cl(D)$, then there is an injective analytic map $\phi:D\to\mathbb{C}$, and a rational function $r\in\mathbb{C}(z)$ such that $f\equiv r\circ\phi$ on $D$.
\end{theorem}

First a topological lemma.

\begin{lemma}\label{lemma: Topological lemma.}
Let $K\subset\mathbb{C}$ be compact, such that $K$ and $K^c$ are both connected, and let $U\subset\mathbb{C}$ be an open set containing $K$.  Then there is a simply connected open set $\mathcal{O}$ containing $K$ such that $cl(\mathcal{O})\subset U$, and such that $\partial\mathcal{O}$ is smooth.
\end{lemma}

This follows from elementary properties of compactness, along with the theorem of Hilbert~\cite{Hi} that the boundary of a simply connected domain may be approximated arbitrarily well by a polynomial lemniscate.

\begin{proof}[Proof of Theorem~\ref{thm: Analytic implies equiv to poly.}.]
To say that $h$ is analytic on the compact set $cl(D)$ is to say that $h$ is analytic on some bounded, open set $\mathcal{O}$ containing $cl(D)$.  We wish to replace $D$ with some larger open set $D_2$ containing $D$ such that, while maintaining the properties of $D$ mentioned in the statement of the theorem, in addition the boundary of $D_2$ has certain nice properties.  The properties which we wish to preserve in $D_2$ are:

\begin{enumerate}
\item[1.]
$cl(D_2)\subset\mathcal{O}$.

\item[2.]
$cl(D_2)$ is simply connected.
\end{enumerate}

Since the zeros of $h$ and of $h'$ are isolated in $\mathcal{O}$, we may expand $D$ slightly to $D_2$ so that $h$ and $h'$ are non-zero on $\partial D_2$.

\begin{enumerate}
\item[3.]
$h$ and $h'$ are non-zero on $\partial D_2$.
\end{enumerate}

By Lemma~\ref{lemma: Topological lemma.}, we may expand $D_2$ slightly so that Properties~1-3 still hold, and now $\partial D_2$ is smooth.  Since $h'$ is non-zero on $\partial D_2$, at each point $w$ on or near $\partial D_2$ the sets $\{z:|h(z)|=|h(w)|\}$ and $\{z:\arg(h(z))=\arg(h(w))\}$ are locally perpendicular to each other at $w$.  Therefore we may again expand $D_2$ in $\mathcal{O}$ slightly so that Properties~1-3 hold, and additionally:

\begin{enumerate}
\item[4.]
$\partial D_2$ is piecewise smooth, and on each smooth segment of $\partial D_2$, either $|h|$ is constant or $\arg(h)$ is constant. 
\end{enumerate}

We will call these segments ``level curve segments'' and ``gradient line segments'' of $\partial D_2$ respectively.  Moreover, since $\partial D$ is compact, we may assume that $\partial D_2$ consists of only finitely many of these segments, and since $h'$ is non-zero on $\partial D_2$, no two level curves or gradient lines of $h$ intersect on $\partial D_2$, so the level curve and gradient line segments of $\partial D_2$ alternate around $\partial D_2$.

If $\partial D_2$ consists of a single level curve of $h$, then the pair $(h,D_2)$ is a generalized finite Blaschke product, and the desired result follows directly from Theorem~\ref{thm: Finite Blaschke product-polynomial conformal equivalence.}.  Otherwise the critical level curves of $h$ in $D_2$ will likely not form a member of $PC_a$.  Our strategy of proof for Theorem~\ref{thm: Analytic implies equiv to poly.} will be to extend the level curves of $h$ in $D_2$ outside of $D_2$ to form members of $P_a$, and further so that the configuration of these extended level curves form a member $\langle\lambda\rangle_{PC_a}$ of $PC_a$.  Theorem~\ref{thm: Polynomial critical level curve config existence.} will then supply us with a polynomial $p$ whose configuration of critical level curves is exactly $\langle\lambda\rangle_{PC_a}$.  If we then restrict the domain of $p$ appropriately, we will see that $h$ is conformally equivalent to the restricted polynomial $p$.  Before making this construction rigorous, we will illustrate this process by the following example.

\begin{figure}[H]
\begin{minipage}[b]{0.45\linewidth}
\centering
\includegraphics[width=\textwidth]{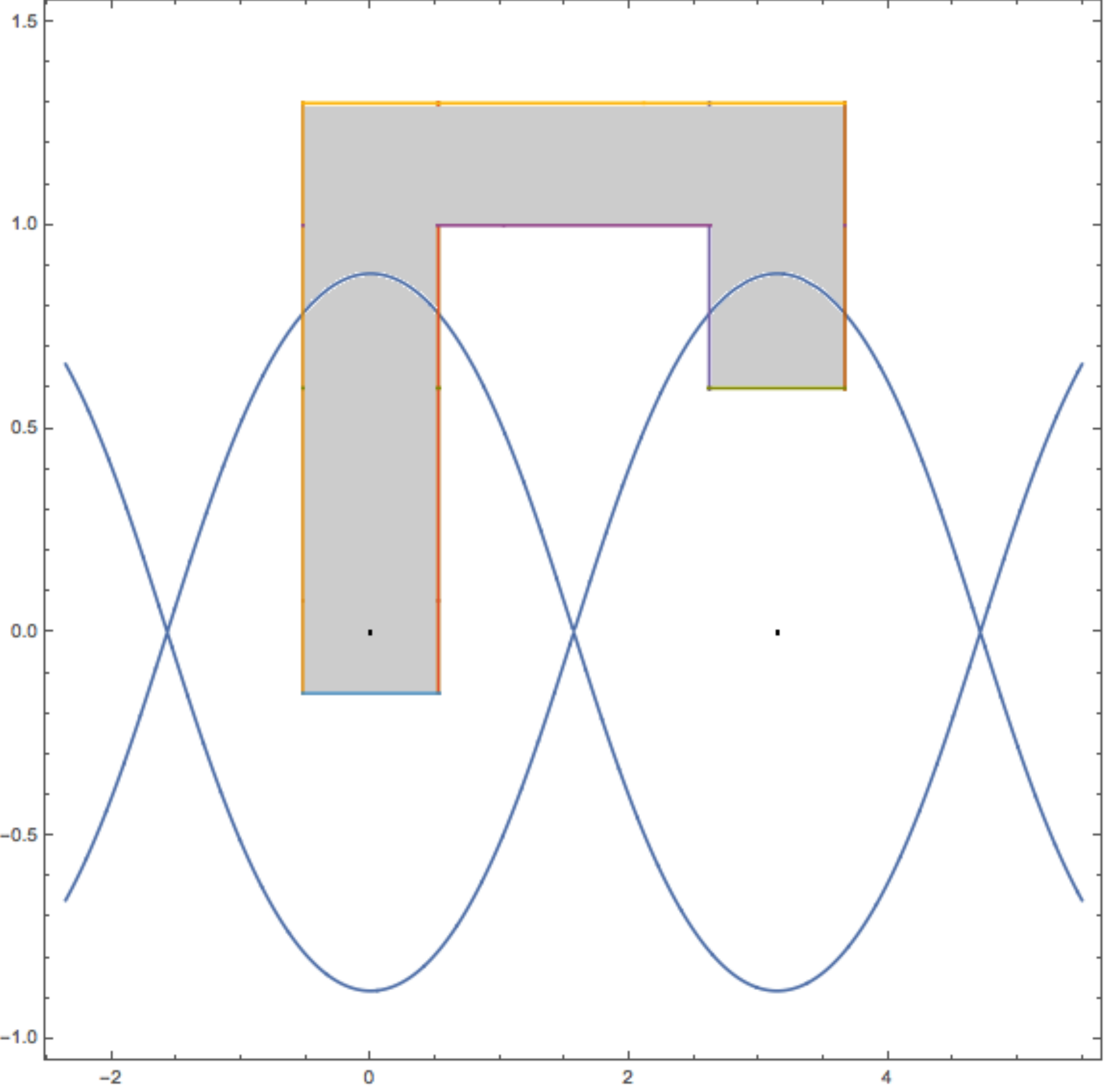}
\caption{The domain $D$}
\label{fig: Domain D.}
\end{minipage}
\hspace{0.5cm}
\begin{minipage}[b]{0.45\linewidth}
\centering
\includegraphics[width=\textwidth]{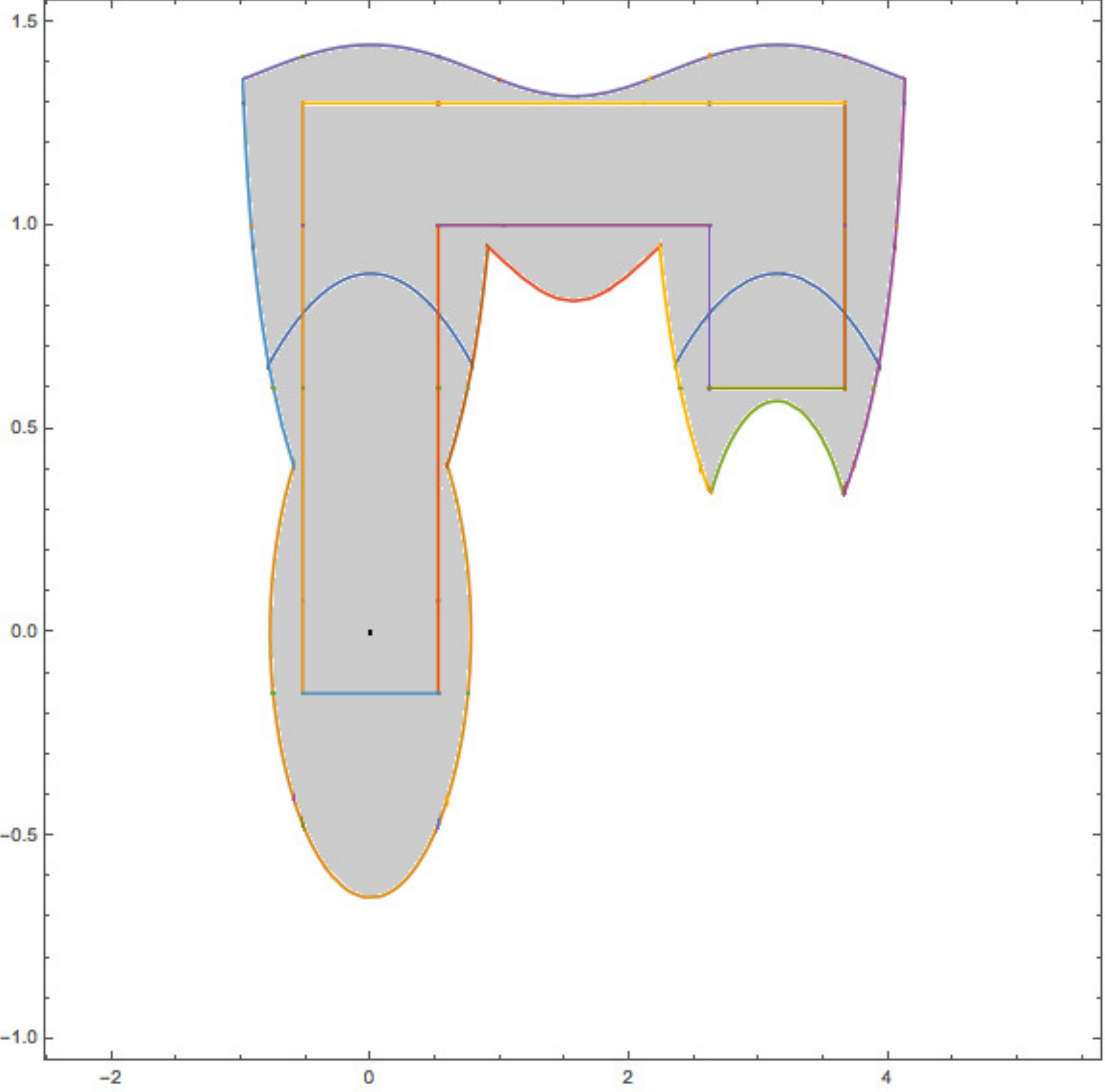}
\caption{The domain $D_2$}
\label{fig: Domain D_2}
\end{minipage}
\end{figure}

\begin{example}
Define $h(z)=\sin(z)$.  All of the critical points of $h$ are contained in a single level curve.  Figure~\ref{fig: Domain D.} depicts the zeros of $h$ along with the critical level curve of $h$, and we let $D$ denote the shaded region in that figure.  Our first step is to expand $D$ slightly to the set $D_2$ depicted in Figure~\ref{fig: Domain D_2}, so that the portions of $\partial D_2$ alternate between stretches of level curve and gradient line of $h$, and disregarding all data about $h$ outside of $D_2$.

We then extend the critical level curves of $h$ out of $D_2$ to form analytic level curve type sets, and extend the level curve segments of $\partial D_2$ similarly, as in Figure~\ref{fig: Level curve segs extended}.  In each bounded face $F$ of one of these extended graphs which does not contain a zero of $h$ already, we identify some point to be an implied zero of $h$ in $F$.  Since two of the level curves are external to each other without any critical level curve separating them, we introduce an implied critical point outside of $D_2$ to account for this.  The implied zero and critical point are depicted in Figure~\ref{fig: Implied points.}.  Finally we find the polynomial whose critical level curves are identical to the configuration we have formed in Figure~\ref{fig: Implied points.}, in this case $p(z)=z^2-\sqrt{3.4}z$.  Figure~\ref{fig: Polynomial region.} depicts the critical level curves of $p$, and if we restrict the domain of $p$ to the shaded region $E$ in that figure, which is the region where $p$ takes the same values that $h$ takes on $D_2$, then $h$ on $D_2$ is conformally equivalent to $p$ on $H$ (ie there is a conformal map $\phi:D_2\to E$ such that $h=p\circ\phi$ on $D_2$).

\begin{figure}[H]
\begin{minipage}[b]{0.45\linewidth}
\centering
\includegraphics[width=\textwidth]{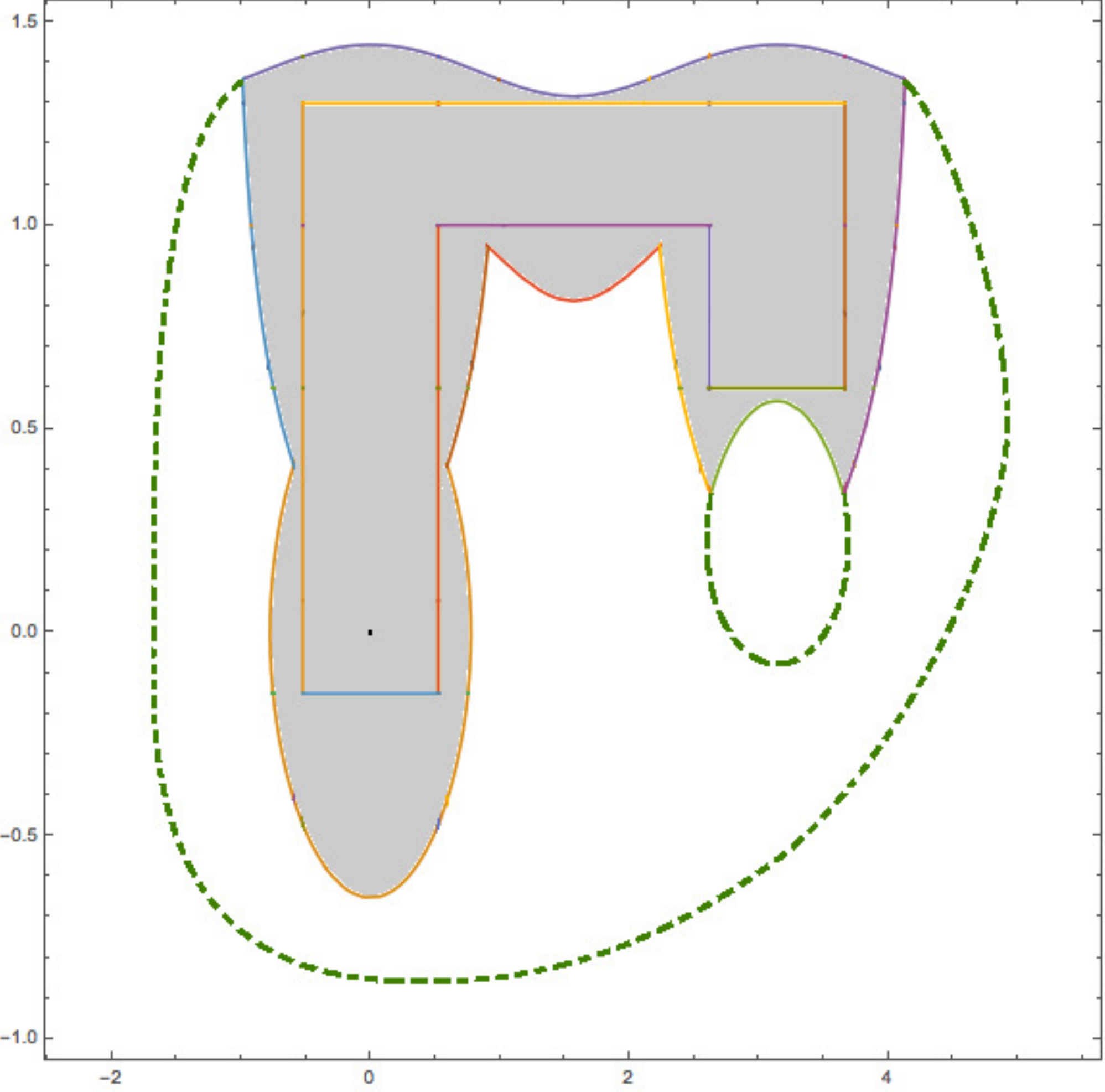}
\caption{Extending the level curve segments of $\partial D_2$}
\label{fig: Level curve segs extended}
\end{minipage}
\hspace{0.5cm}
\begin{minipage}[b]{0.45\linewidth}
\centering
\includegraphics[width=\textwidth]{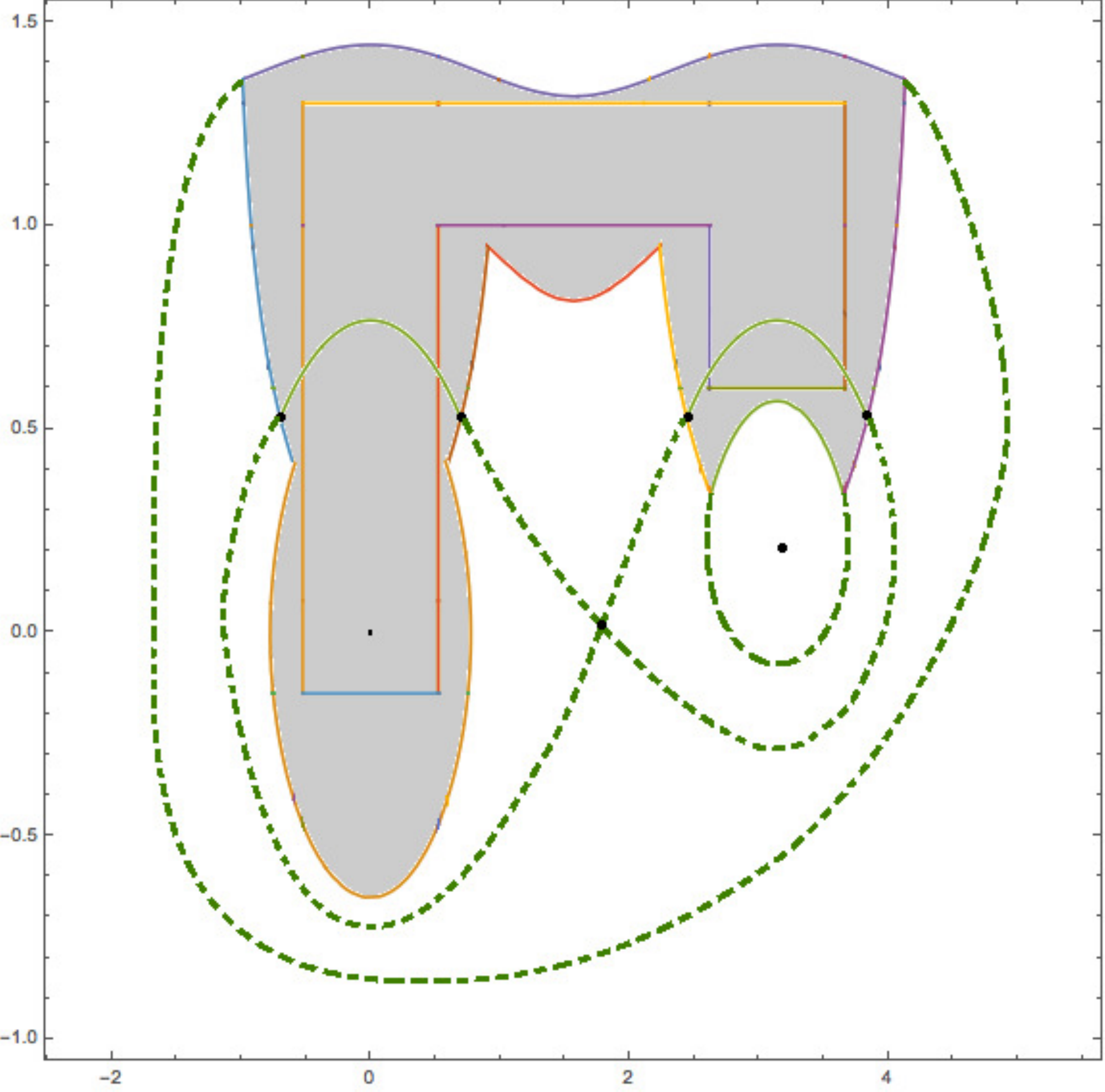}
\caption{The implied zero and critical point of $h$}
\label{fig: Implied points.}
\end{minipage}
\end{figure}

\begin{figure}[H]
\centering
\includegraphics[width=0.45\textwidth]{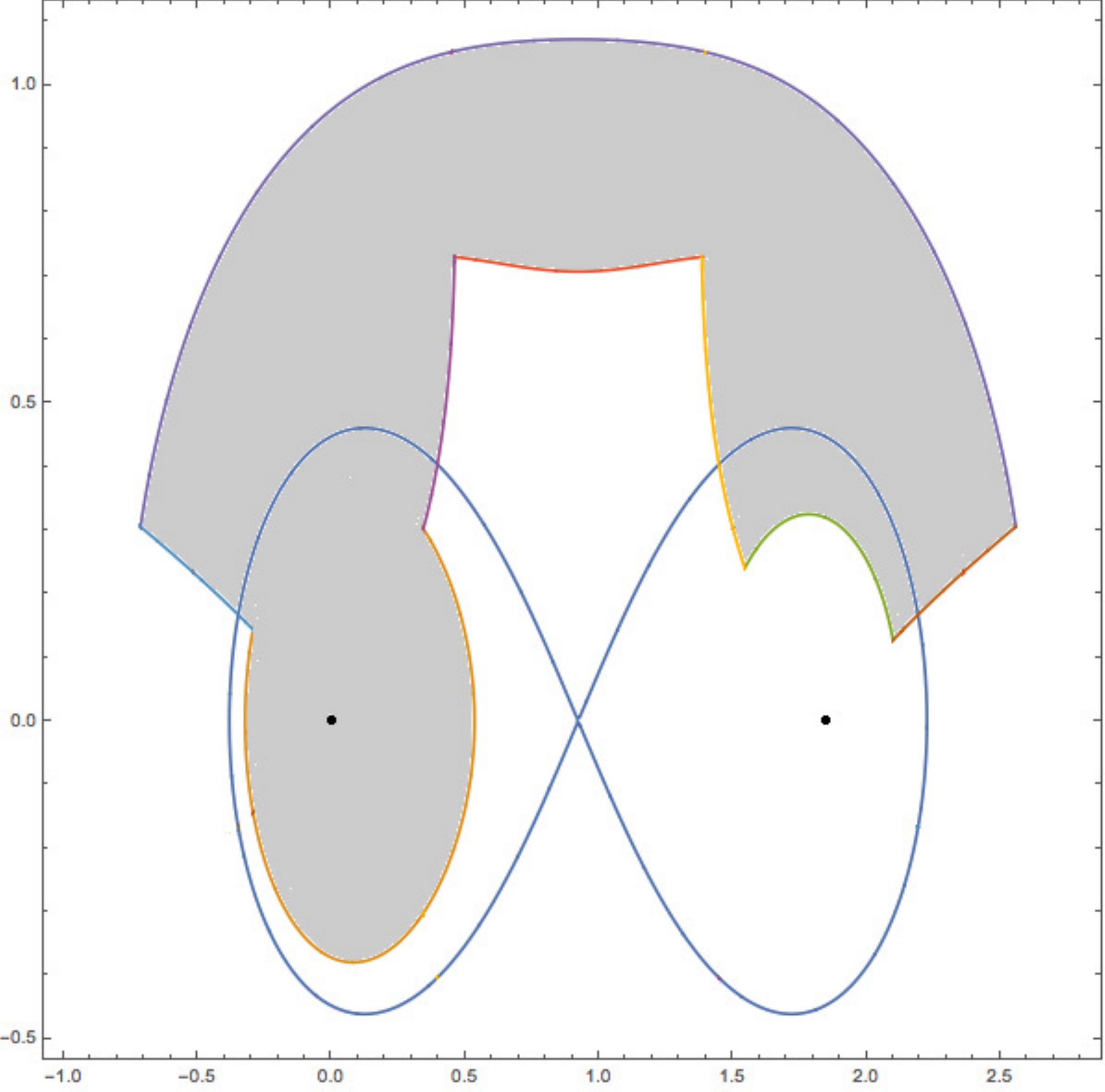}
\caption{Region for $p(z)=z^2-\sqrt{3.4}z$}
\label{fig: Polynomial region.}
\end{figure}

\end{example}

Returning to the proof of Theorem~\ref{thm: Conformal equivalence, most general.}, our extension of the level curves of $h$ will occur in three phases.  First we will extend (if necessary) the level curves of $h$ which contain the critical points of $h$ in $D_2$ (we will call these level curves the \textit{explicit critical level curves} of $h$ in $cl(D_2)$).  Second we will extend the level curves of $h$ in $cl(D_2)$ which contain level curve segments of $\partial D_2$.  Finally, by inspecting the already extended (and unextended) level curves of $h$ in $cl(D_2)$, we will extend some level curves of $h$ in $cl(D_2)$ to account for any ``implied critical points'' of $h$.  In order to explain this notion of an implied critical point of $h$, consider the following.  Suppose that $h$ has several distinct zeros in $D_2$, but no critical points in $D_2$.  Suppose further that we extend the level curves of $h$ in $cl(D_2)$ to form some member of $PC_a$, $\langle\lambda\rangle_{PC_a}$.  Theorem~\ref{thm: Polynomial critical level curve config existence.} implies that $\langle\lambda\rangle_{PC_a}=\Pi(p,G_p)$ for some polynomial $p$ (where $G_p$ is the set $\{z:|p(z)|<\epsilon\}$ for some $\epsilon>0$), and this polynomial $p$ must in turn have at least as many distinct zeros in $G_p$ as $h$ has in $D_2$.  However, it is not hard to show that any polynomial $p$ with several distinct zeros must have critical points which are not also zeros of $p$.  Therefore the configuration $\langle\lambda\rangle_{PC_a}$ must contain critical points (ie. vertices of graph members of $P_a$) which are not zeros (ie. single point members of $P_a$), and thus in our extension of the level curves of $h$ in $cl(D_2)$, we must, in some sense, produce these critical points.  These critical points which we produce with our level curve extensions outside of $D_2$ will be called \textit{implied critical points} of $h$.

In order to make rigorous the notion of the extension of a level curve of $h$ outside of $D_2$, we will need several definitions.  We begin by categorizing the gradient line and level curve segments of $\partial D_2$ as follows.

\begin{definition}
Let $E$ be one of the gradient line segments of $\partial D_2$.

\begin{itemize}
\item
If $\arg(h(z))$ is increasing as $z$ crosses $E$ from inside $D_2$ moving outwards, then we call $E$ a \textit{left edge} of $\partial D_2$.
\item
If $\arg(h(z))$ is decreasing as $z$ crosses $E$ from inside $D_2$ moving outwards, then we call $E$ a \textit{right edge} of $\partial D_2$.
\end{itemize}

Now let $E$ be one of the level curve segments of $\partial D_2$.

\begin{itemize}
\item
If $|h(z)|$ is increasing as $z$ crosses $E$ from inside $D_2$ moving outwards, then we call $E$ a \textit{top edge} of $\partial D_2$.
\item
If $|h(z)|$ is decreasing as $z$ crosses $E$ from inside $D_2$ moving outwards, then we call $E$ a \textit{bottom edge} of $\partial D_2$.
\end{itemize}
\end{definition}

We also categorize the corners of $\partial D_2$ as follows.

\begin{definition}
For $u\in\partial D_2$, we say that $u$ is a corner of $\partial D_2$ if $u$ is contained in both a gradient line segment of $\partial D_2$ and a level curve segment of $\partial D_2$.  In this case, we make the following definition.

\begin{itemize}
\item
If $cl(D_2)$ is locally convex at $u$, then we call $u$ an \textit{outside corner} of $\partial D_2$.
\item
If $cl(D_2)$ is not locally convex at $u$, then we call $u$ an \textit{inside corner} of $\partial D_2$.
\end{itemize}
\end{definition}

For a point $u$ in one of the gradient line segments of $\partial D_2$, which is not an inside corner of $\partial D_2$, we wish to have a well-defined way of selecting the other end point of a prospective extension out of $cl(D_2)$ of the level curve of $h$ which contains $u$, so we make the following definition.

\begin{definition}
Let $u\in\partial D_2$ be a point in one of the gradient line segments of $\partial D_2$ which is not an inside corner of $\partial D_2$.  Let $E$ denote this segment of $\partial D_2$ and let $\epsilon>0$ denote the value $|h(u)|$.  If $E$ is a right edge of $\partial D_2$ (respectively left edge of $\partial D_2$), let $v$ denote the first point in $\partial D_2$ after $u$ in the negative (respectively positive) direction such that $|h(v)|=\epsilon$.  We call $v$ the \textit{neighbor of $u$}, and write $v=\mathcal{N}(u)$.
\end{definition}

In the definition of left and right edges of $\partial D_2$ above, one might equivalently have said that a gradient line segment $E$ of $\partial D_2$ is a left edge of $\partial D_2$ if $|h|$ is decreasing as $E$ is traversed with positive orientation around $\partial D_2$, and $E$ is a right edge of $\partial D_2$ if $|h|$ is increasing as $E$ is traversed with positive orientation around $\partial D_2$.  With this equivalent definition in mind it is not hard to see that, if $u$ is in a gradient line segment $E$ of $\partial D_2$, and $u$ is not an inside corner of $\partial D_2$, then the following holds.

\begin{itemize}
\item
If $E$ is a left edge (respectively right edge) of $\partial D_2$, then $\mathcal{N}(u)$ is in a right edge (respectively left edge) of $\partial D_2$.

\item
$\mathcal{N}(u)$ is not an inside corner of $\partial D_2$.

\item
$\mathcal{N}(\mathcal{N}(u))=u$.
\end{itemize}

\begin{definition}
Let $u\in\partial D_2$ be any point in a gradient line segment of $\partial D_2$ which is not an inside corner of $\partial D_2$.  If $u$ is in a left edge of $\partial D_2$, let $L_u$ denote the segment of $\partial D_2$ with initial point $u$ and final point $\mathcal{N}(u)$ (with respect to positive orientation around $\partial D_2$).  If $u$ is in a right edge of $\partial D_2$, let $L_u$ denote the segment of $\partial D_2$ with initial point $\mathcal{N}(u)$ and final point $u$ (with respect to positive orientation around $\partial D_2$).
\end{definition}

Note that by the symmetry of the definition of a neighbor point, $L_u=L_{\mathcal{N}(u)}$.  Also, by the maximum modulus theorem and the open mapping theorem, if $\epsilon=|f(u)|$, then for all $z\in L_u$, $|f(z)|<\epsilon$.  We will now define the way that, for a level curve $\Lambda$ of $h$ in $D_2$ which intersects $\partial D_2$ in a gradient line segment, but not at an inside corner, we will extend $\Lambda$ out of $D_2$ to form an analytic level curve type set.

\begin{definition}
Let $u\in\partial D_2$ be any point in a gradient line segment of $\partial D_2$ which is not an inside corner of $\partial D_2$.  We define $\gamma_u$ (up to homotopy) to be the simple path with initial point $u$ and final point $\mathcal{N}(u)$ which is contained in ${cl(D_2)}^c$ (except of course at its end points) such that $L_u$ is not adjacent to the unbounded face of $cl(D_2)\cup\gamma_u$.  We call $\gamma_u$ an \textit{extension path} and, in particular, we say that $\gamma_u$ extends the level curve of $h$ in $cl(D_2)$ which contains $u$.
\end{definition}

Again note that, by symmetry, $\gamma_u=\gamma_{\mathcal{N}(u)}$.  We will now construct the full extension of a given level curve of $h$ in $cl(D_2)$.  Let $\Lambda$ be any level curve of $h$ in $cl(D_2)$, and let $\epsilon\geq0$ denote the value $|h|$ takes on $\Lambda$.  If $\Lambda$ does not contain any point in a gradient line segment of $\partial D_2$ which is not an inside corner of $\partial D_2$, then by the full extension of $\Lambda$ we just mean $\Lambda$ itself.  To make it easier to refer to different level curves of $h$ in $cl(D_2)$, we make the following definition.

\begin{definition}
For any point $u$ in $cl(D_2)$, define $\Lambda_u$ to be the level curve of $h$ in $cl(D_2)$ which contains $u$.
\end{definition}

Suppose that $\Lambda$ contains some point $u$ in $\partial D_2$ which is in a gradient line segment of $\partial D_2$, which is not an inside corner of $\partial D_2$, and to which we have not already attached an extending path.  We now replace $\Lambda$ with $\Lambda\cup\gamma_u\cup\Lambda_{\mathcal{N}(u)}$.  That is, we extend $\Lambda$ with $\gamma_u$ as well as with the level curve of $h$ in $cl(D_2)$ which contains the other end point of $\gamma_u$.  Since there are only finitely many points in gradient line segments of $\partial D_2$ at which $|h|$ takes the value $\epsilon$, if we iterate this process it will terminate in finitely many steps.  That is, after finitely many steps $\Lambda$ will not contain any point in a gradient line segment of $\partial D_2$ which is not an inside corner, and to which we have not already attached an extending path.  The graph obtained at the final step in this process we call the \textit{full extension} of $\Lambda$.

For any points $u$ and $u'$ contained in a gradient line segment of $\partial D_2$, by the definition of a neighbor point, $u'$ is contained in $L_u$ if and only if $\mathcal{N}(u')$ is contained in $L_u$, and if $u'$ is contained in $L_u$, then $L_{u'}\subset L_u$.  Therefore we may always select $\gamma_u$ and $\gamma_{u'}$ so that they do not intersect, and therefore while extending $\Lambda$ as described above, we may choose our extension paths so that none of the extension paths intersect each other and, moreover, if $\Lambda'$ is any other level curve of $h$ in $cl(D_2)$, the full extensions of $\Lambda$ and $\Lambda'$ are either equal or non-intersecting.

We will now show that the full extension $\Lambda$ of any level curve of $h$ in $cl(D_2)$ has the following properties.

\begin{enumerate}
\item\label{item: Connected finite graph.}
$\Lambda$ is a connected finite graph.

\item\label{item: Evenly many edges.}
There are evenly many, and more than two, edges of $\Lambda$ incident to each vertex of $\Lambda$ (where we count an edge twice if both of its end points are at a given vertex).

\item\label{item: Adjacent to unbounded.}
Each edge of $\Lambda$ is incident to the unbounded face of $\Lambda$.
\end{enumerate}

That is, we wish to show that $\Lambda$ is an analytic level curve type set.  Item~\ref{item: Connected finite graph.} should be immediately obvious from the construction of $\Lambda$.  In order to show Item~\ref{item: Evenly many edges.}, let $v$ be a vertex of $\Lambda$ (if $\Lambda$ has any vertices).  By the construction of a full extension, there are no vertices of $\Lambda$ outside of $cl(D_2)$, so $v\in cl(D_2)$.  Moreover, if $v$ were in $\partial D_2$, by the construction of a full extension above, all but one of the edges of $\Lambda$ which are incident to $v$ must extend into $D_2$.  This would imply that $v$ is a critical point of $h$, since level curves of $h$ are smooth away from critical points of $h$.  This cannot occur however, as $h'\neq0$ on $\partial D_2$.  We conclude therefore that $v\in D_2$, and that $v$ is a critical point of $h$.  Let $k>0$ denote the multiplicity of $v$ as a zero of $h'$.  Since $h$ is $(k+1)$-to-$1$ in a neighborhood of $v$, there are $2(k+1)$ edges of $\Lambda$ meeting at $v$ (counting an edge twice if both its end points are at $v$).  This establishes Item~\ref{item: Evenly many edges.}.

Let $E$ denote one of the edges of $\Lambda$.  As discussed above, the vertices of $\Lambda$ (and thus the end points of $E$) are critical points of $h$ in $D_2$, so there are points in $E$ which are contained in $D_2$.  Let $w$ denote one of these points.  The open mapping theorem implies that there are points arbitrarily close to $w$ at which $|h|$ takes values greater than $\epsilon$.  Therefore if we show that $|h(z)|<\epsilon$ for each point $z$ in $cl(D_2)$ which is also contained in a bounded face of $\Lambda$, that implies that $E$ is adjacent to the unbounded face of $\Lambda$.  

Let $\Lambda_0$ denote the original level curve which was extended to form the full extension $\Lambda$.  Let $N\geq0$ denote the number of steps required to form $\Lambda$ from $\Lambda_0$, and for each $i\in\{1,\ldots,N\}$, let $\Lambda_i$ denote the graph obtained after the $i^{\text{th}}$ step (thus the fully extended graph $\Lambda$ equals $\Lambda_N$).  We will prove the desired result by induction on $i$.  That is, by applying induction to the statement $\mathcal{J}(i)$ below.
\[
\mathcal{J}(i):\text{For each }z\in cl(D_2)\text{ which is contained in a bounded face of }\Lambda_i,|h(z)|<\epsilon.
\]

We first show $\mathcal{J}(0)$.  Let $F$ denote one of the bounded faces of $\Lambda_0$.  Since $\Lambda_0$ consists only of a level curve of $h$ in $cl(D_2)$, $F$ is contained entirely in $cl(D_2)$, so the maximum modulus theorem implies that $|h(z)|<\epsilon$ for each $z\in F$.  This establishes $\mathcal{J}(0)$.

Now select $i\in\{1,\ldots,N\}$, and suppose that $\mathcal{J}(i-1)$ holds.  Let $u_i$ be the point (or one of the points) in $\Lambda_{i-1}$ such that $\gamma_{u_i}$ is the extension path added to $\Lambda_{i-1}$ in the formation of $\Lambda_i$.  Let $F_1$ denote some bounded face of $\Lambda_i$, and let $F_2$ denote one of the components of $F_1\cap cl(D_2)$.    Observe first that if $F_1$ is a bounded face of $\Lambda_{i-1}$ alone, then the desired result follows from the induction assumption, and if $F_1$ is a bounded face of $\Lambda_{\mathcal{N}(u_i)}$ alone, then the desired result follows from the same reasoning found in our proof of the base case (ie $\mathcal{J}(0)$).  Therefore let us suppose that $\partial F_1$ is not contained entirely in either $\Lambda_{i-1}$ or $\Lambda_{\mathcal{N}(u_i)}$.

By definition of a level curve of an analytic function, the maximum modulus theorem, and the induction assumption, either $\Lambda_{\mathcal{N}(u_i)}\subset\Lambda_{i-1}$ or $\Lambda_{\mathcal{N}(u_i)}$ and $\Lambda_{i-1}$ are mutually exterior (that is, each is contained in the unbounded face of the complement of the other).  However, since $\gamma_{u_i}$ is a simple path with end points in $\Lambda_{i-1}$ and $\Lambda_{\mathcal{N}(u_i)}$, if $\Lambda_{i-1}$ and $\Lambda_{\mathcal{N}(u_i)}$ are mutually exterior then $\gamma_{u_i}$ is not incident to any bounded face of $\Lambda_i$.  But $F_1$ is not a bounded face of $\Lambda_{i-1}$ or of $\Lambda_{\mathcal{N}(u_i)}$ alone, so $F_1$ must be incident to $\gamma_{u_i}$.  We conclude that $\Lambda_{i-1}$ and $\Lambda_{\mathcal{N}(u_i)}$ cannot be mutually exterior, and thus $\Lambda_{\mathcal{N}(u_i)}\subset\Lambda_{i-1}$, and thus $\Lambda_i=\Lambda_{i-1}\cup\gamma_{u_i}$.

If $F_2$ (the component of $F_1\cap cl(D_2)$) were contained in a bounded face of $\Lambda_{i-1}$, the desired result would hold by the induction assumption.  Thus $F_2$ is contained in the unbounded face of $\Lambda_{i-1}$.  The only segment of $\partial D_2$ contained in a bounded face of $\Lambda_i$ but in the unbounded face of $\Lambda_{i-1}$, is $L_{u_i}$.  Since $\partial F_2\subset\Lambda_i\cup\partial D_2$, it follows that $\partial F_2\subset\Lambda_i\cup L_{u_i}$.  Since $|h|<\epsilon$ on $L_{u_i}$, and $|h|=\epsilon$ at any point in $\Lambda_i\cap cl(D_2)$, $|h|\leq\epsilon$ on $\partial F_2$.  The desired result now follows from the maximum modulus theorem.  This establishes Item~\ref{item: Adjacent to unbounded.}, and we thus conclude that the full extension $\Lambda$ is an analytic level curve type set.

We will now select certain of the level curves of $h$ in $cl(D_2)$ to fully extend. Let $\mathbb{A}$ denote the collection of the full extensions of the following level curves $h$ in $cl(D_2)$.

\begin{itemize}
\item
Full extensions of all explicit critical level curves of $h$ in $cl(D_2)$.
\item
Full extensions of level curves of $h$ in $cl(D_2)$ which intersect a level curve segment of $\partial D_2$.
\item
Zeros of $h$ in $cl(D_2)$.
\end{itemize}

Since there are only finitely many of each of these classes of level curves of $h$ in $cl(D_2)$, $\mathbb{A}$ is a collection of finitely many analytic level curve type sets.

If $\Lambda$ is one of the full extensions in $\mathbb{A}$, and $\gamma$ is one of the extension paths used to construct $\Lambda$, we will now define a notion of the change in $\arg(h)$ along $\gamma$, which we will call $\Delta_{\arg}(\gamma)$.  Of course since $\gamma$ extends outside of $cl(D_2)$, $h$ need not be defined on $\gamma$.  The idea is to ask, if $h$ were defined on $\gamma$, what would the change in $\arg(h)$ along $\gamma$ be assuming that $h$ is as simple as possible (ie. $h$ does not have any zeros outside of $D_2$ whose presence are not already implied in some way by facts about $h$ in $D_2$).  We will then define $\Delta_{\arg}(\gamma)$ to be this quantity.

To begin with, for any segment $L\subset\partial D_2$, let $\Delta_{\arg}(L)$ denote the change in $\arg(h(z))$ as $z$ traverses $L$ with positive orientation with respect to $\partial D_2$.  We will define the quantities $\Delta_{\arg}(\gamma)$ for the different extension paths $\gamma$ used to construct the members of $\mathbb{A}$ recursively, working ``inside out'' as follows.

Let $\Lambda$ be a member of $\mathbb{A}$ which contains a point $u$ which is in a gradient line segment of $\partial D_2$, and which is not an inside corner of $\partial D_2$.  Since $\gamma_u=\gamma_{\mathcal{N}(u)}$, we may assume without loss of generality that $u$ is in a left edge of $\partial D_2$.  Let $F_u$ denote the bounded component of $(cl(D_2)\cup\gamma_u)^c$.  We begin by assuming that $F_u$ does not contain any extension path used to construct any other member of $\mathbb{A}$ (the ``base case'' of our recursive definition of $\Delta_{\arg}(\gamma)$).  The boundary of $F_u$ consists of $\gamma_u$ and $L_u$, so we define $\Delta_{\arg}(\gamma_u)\in(0,\infty)$ to be the least number such that $\Delta_{\arg}(\gamma_u)-\Delta_{\arg}(L_u)\geq0$ and $\Delta_{\arg}(\gamma_u)-\Delta_{\arg}(L_u)$ is an integer multiple of $2\pi$.  (Note that since $F_u\subset cl(D_2)^c$, the change in $\arg(h)$ along $L_u$ as $L_u$ is traversed with positive orientation with respect to $F_u$ is $-\Delta_{\arg}(L_u)$.)  $\Delta_{\arg}(\gamma_u)$ has been chosen to be the smallest number so that, if $h$ could be extended to an analytic function on $F_u$, with $\arg(h)$ increasing along $\gamma_u$, then the net change in $\arg(h)$ along $\gamma_u$ might be $\Delta_{\arg}(\gamma_u)$.

Now let us suppose that $F_u$ contains some extension paths which were used in the construction of members of $\mathbb{A}$.  Let $\gamma_1,\ldots,\gamma_M$ denote these extension paths.  Suppose recursively that $\Delta_{\arg}(\gamma_i)$ has been defined for each $1\leq i\leq M-1$.  Let ${F_u}'\subset F_u$ denote the bounded face of $\left(cl(D_2)\cup\gamma_u\cup\ds\bigcup_{i=1}^M\gamma_i\right)^c$ to which $\gamma_u$ is adjacent.  Let $L$ denote the path $\partial {F_u}'\setminus\gamma_u$, and let $\Delta_{\arg}(L)$ denote the total change in $\arg(h(z))$ as $z$ traverses $L$ with positive orientation (with respect to ${F_u}'$), using $\Delta_{\arg}(\gamma_i)$ when calculating $\Delta_{\arg}(L)$ if $\gamma_i\subset L$.  We now define $\Delta_{\arg}(\gamma_u)\in(0,\infty)$ to be the least number such that $\Delta_{\arg}(\gamma_u)-\Delta_{\arg}(L)\geq0$ and $\Delta_{\arg}(\gamma_u)-\Delta_{\arg}(L)$ is an integer multiple of $2\pi$.

With the quantity $\Delta_{\arg}(\gamma)$ defined for all extension paths $\gamma$ used in the construction of the members of $\mathbb{A}$, we can now determine the locations of any implied zeros and critical points of $h$.  Let $F$ denote one of the bounded components of the set $\left(cl(D_2)\cup\ds\bigcup_{\Lambda\in\mathbb{A}}\Lambda\right)^c$.  $F$ is of course contained in ${cl(D_2)}^c$, and thus does not contain any zero of $h$ in $cl(D_2)$.  However the members of $\mathbb{A}$ are meant to represent the possible level curve structure of an analytic function which has the same level curve structure as that of $h$ in $cl(D_2)$.  Let $(f,G)$ denote a generalized finite Blaschke product whose level curve structure extends the level curve structure of $h$ in $cl(D_2)$, and such that the extension paths which we have used to make the full extensions found in $\mathbb{A}$ are also level curves of $f$ in $G$ (if such a generalized finite Blaschke product $(f,G)$ exists).  Then $\Delta_{\arg}(\partial F)$, the change in $\arg(f(z))$ as $z$ traverses $\partial F$ with positive orientation, should be equal to $2\pi$ times the number of zeros of $f$ in $F$.  Therefore we say that there are $\Delta_{\arg}(\partial F)/2\pi$ implied zeros of $h$ in $F$.  Define $n\colonequals\Delta_{\arg}(\partial F)/2\pi$.

If $n=0$, then we leave $F$ alone.  Suppose now that $n>0$.  Since the members of $\mathbb{A}$ have extended all level curve segments of $\partial D_2$, either $\partial F$ consists of a single level curve portion of $\partial D_2$ along with a single extension path $\gamma$, or $\partial F$ consists alternatingly of exactly two level curve segments of $\partial D_2$ and two gradient line segments of $\partial D_2$.

Suppose that $\partial F$ consists of a single portion of $\partial D_2$ and a single extension path.  Then we choose some point fixed point in $F$ which we call an implied zero of $f$ with multiplicity $n$.

Suppose now that the second possibility obtains, that $\partial F$ alternates between exactly two level curve segments of $\partial D_2$ and two gradient line segments of $\partial D_2$.  (If $\partial F$ contains an extension path $\gamma$, then we view $\gamma$ as a level curve segment, as it is meant to represent a possible segment of level curve of $h$ extending out of $D_2$.)  It follows that of the two gradient line segments of $\partial D_2$ which form $\partial F$, one is part of a left edge of $\partial D_2$ and one is part of a right edge of $\partial D_2$.

Select some point $u$ in the left edge gradient line segment of $\partial D_2$ which forms part of $\partial F$ such that $u$ is not a corner of $\partial F$.  Since the extension path $\gamma_u$ may be selected so as not to intersect any other extension path already used to form the members of $\mathbb{A}$, it follows that $\mathcal{N}(u)$ is contained in the other gradient line segment of $\partial D_2$ which forms $\partial F$.  Define $\epsilon\colonequals|h(u)|$.  $\gamma_u$ dissects $F$ into two pieces.  Let $F_1$ denote the piece of $F\setminus \gamma_u$ such that $|h(z)|\leq\epsilon$ on $\partial F_1$, which is the left piece as $\gamma_u$ is traversed from $u$ to $\mathcal{N}(u)$ since $h$ is analytic.  Let $L$ denote the segment of $\partial F$ with beginning point $\mathcal{N}(u)$ and end point $u$ (with respect to positive orientation around $F$) which is contained in $\partial F_1$.  Let $\Delta_{\arg}(L)$ denote the change in $\arg(h(z))$ as $z$ traverses $L$ from $u$ to $\mathcal{N}(u)$.  Then we define $\Delta_{\arg}(\gamma_u)\colonequals\Delta_{\arg}(L)$.  Thus the net change in $\arg(h(z))$ as $z$ traverses $\partial F_1$ is $\Delta_{\arg}(\gamma_u)-\Delta_{\arg}(L)=0$.  Let $F_2$ denote the other component of $F\setminus\gamma_u$.  We still have that the net change in $\arg(h(z))$ as $z$ traverses $\partial F_2$ equals $2\pi n$, which we will now account for with an implied zero of $h$ in $F_2$.

Select some point $u'$ in $\gamma_u$ which is not an end point of $\gamma_u$.  Choose some fixed value $\alpha\in(\arg(h(u)),\arg(h(u))+\Delta_{\arg}(\gamma_u))$, and define $\arg(h(u'))$ to be the number $\alpha$.  Join a circle $C$ contained in $F_2$ to $\gamma_u$ at $u'$.  In the bounded face of this circle select some fixed point which we will now call an implied zero of $h$ of multiplicity $n$.  Include this implied zero of $h$ in $\mathbb{A}$.  Define $\Delta_{\arg}(C)$ to be $2\pi n$, and replace $\gamma_u$ with $\gamma_u\cup C$.  Let $F_3$ denote the bounded face of $C$, and let $F_4$ denote the segment of $F_2$ outside of $C$.  We now have that $F$ is decomposed into $F_1$, $F_3$, and $F_4$.  The net change in $\arg(h(z))$ as $z$ traverses $\partial F_1$ or $\partial F_4$ is $0$, and these regions contain no zeros of $h$.  The net change in $\arg(h(z))$ as $z$ traverses $\partial F_3$ (which equals $C$) is $2\pi n$, and this region now contains an implied zero of $h$ of multiplicity $n$.  Form the rest of the full extension of the level curve of $h$ which contains the point $u$ in the usual way and include it (along with the joined arc $C$) in $\mathbb{A}$.

It is worth noting here that since the circle $C$ was chosen to reside in the region $F_2$ such that $|h|\geq\epsilon$ on $\partial F_2$, the maximum modulus theorem implies that $C$ resides in the unbounded face of the the full extension of $\Lambda_u$.  Therefore since this full extension is an analytic level curve type set, the graph obtained by joining $C$ to this full extension is an analytic critical level curve type set as well.

By performing this selection of implied zeros for each face $F$ of $cl(D_2)\cup\mathbb{A}$ for which $\Delta_{\arg}(\partial F)\neq0$, we account for all the implied zeros of $h$.  That is, having done this process, if $F$ is a face of $cl(D_2)\cup\mathbb{A}$, then $\Delta_{\arg}(\partial F)=2\pi n$ for some non-negative integer $n$, and if $n>0$ then $F$ contains a single distinct implied zero of $h$ with multiplicity $n$.

We now describe the phenomenon of implied critical points of $h$, and describe how we take account of them by adding certain extended level curves to $\mathbb{A}$.  We begin by defining an ordering on the members of $\mathbb{A}$.

\begin{definition}
For any $\Lambda_1,\Lambda_2\in\mathbb{A}$, if $\Lambda_1$ is contained in one of the bounded faces of $\Lambda_2$ then we write $\Lambda_1\prec\Lambda_2$.  For any domain $G$ in $\mathbb{C}$, if $\Lambda_1\in\mathbb{A}$ with $\Lambda_1\subset G$, and there is no other $\Lambda_2\in\mathbb{A}$ such that $\Lambda_1\prec\Lambda_2\subset G$, then we say that $\Lambda_1$ is \textit{$\prec$-maximal} in $G$.
\end{definition}

In view of Theorem~\ref{thm: Level curves separated by critical level curve.}, if the extensions of the level curves of $h$ in $cl(D_2)$ are to form a member of $PC_a$, then for each $\Lambda\in\mathbb{A}$, and bounded face $G$ of $\Lambda$, there should be a unique member of $\mathbb{A}$ in $G$ which is $\prec$-maximal in $G$.  From our work above accounting for implied zeros of $h$, it follows that each such $G$ contains at least one member of $\mathbb{A}$, and thus (because $\mathbb{A}$ is finite) at least one $\prec$-maximal member of $\mathbb{A}$.

If some such $G$ contains several $\prec$-maximal members of $\mathbb{A}$, then in order to ensure that the result of Theorem~\ref{thm: Level curves separated by critical level curve.} holds for the members of $\mathbb{A}$, we will extend one of the level curves of $h$ in $G$ to form an analytic critical level curve type set such that several of the $\prec$-maximal members of $\mathbb{A}$ in $G$ are contained in different of its faces.  The vertex of this new extension will represent an implied critical point of $h$ in $cl(D_2)$.

Suppose that $G$ is some fixed bounded face of a member of $\mathbb{A}$ which contains several $\prec$-maximal members of $\mathbb{A}$.  Let $\epsilon>0$ denote the value that $|h|$ takes on $\Lambda$.  Let $\Lambda_1,\ldots,\Lambda_N$ denote the members of $\mathbb{A}$ which are $\prec$-maximal in $G$.  For each $i\in\{1,\ldots,N\}$, define $\epsilon_i$ to be the value $|h|$ takes on $\Lambda_i$, and assume that $\Lambda_1,\ldots,\Lambda_N$ are ordered so that $\epsilon_1\leq\cdots\leq\epsilon_N$.  Then, again for each $i\in\{1,\ldots,N\}$, if $\widetilde{\epsilon}\in(\epsilon_N,\epsilon)$, then $\widetilde{\epsilon}\in(\epsilon_i,\epsilon)$ so, by the continuity of $|h|$, there is some level curve of $h$ in $G$ whose full extension contains $\Lambda_i$ in its bounded face, and on which $|h|$ takes the value $\widetilde{\epsilon}$.  By the maximum modulus theorem this level curve is unique, and we denote its full extension by $\widetilde{\Lambda_i}$ (which, we note, does depend on the choice of $\widetilde{\epsilon}\in(\epsilon_N,\epsilon)$).  (Note that it may be that $\widetilde{\Lambda_i}=\widetilde{\Lambda_j}$ even if $i\neq j$.)  Since the level curves of $h$ vary continuously in $cl(D_2)$, and by the definition of the of extension paths, if $\widetilde{\epsilon}\in(\epsilon_N,\epsilon)$ is sufficiently close to $\epsilon_N$, then $\widetilde{\Lambda_N}$ approximates $\Lambda_N$, and consists of a single simple closed path in $G$ containing $\Lambda_N$ in its bounded face and each $\Lambda_i$ (for $i\in\{1,\ldots,N-1\}$) in its unbounded face.  On the other hand, again by the continuity of the level curves of $h$ in $D_2$ and the definition of extension paths, if $\widetilde{\epsilon}\in(\epsilon_N,\epsilon)$ is sufficiently close to $\epsilon$, then $\widetilde{\Lambda_N}$ approximates $\partial G$, and contains $\Lambda_i$ in its bounded face for each $i\in\{1,\ldots,N\}$.

Therefore if we begin $\widetilde{\epsilon}\in(\epsilon_N,\epsilon)$ close to $\epsilon_N$, and allow $\widetilde{\epsilon}$ to increase towards $\epsilon$, eventually $\widetilde{\Lambda_N}$ will collide with some $\widetilde{\Lambda_i}$ with $i\neq N$.  If this collision were to happen in $cl(D_2)$, then the collision point of the level curves would be a vertex of a level curve of $h$ in $cl(D_2)$, and thus an explicit critical point of $h$ in $cl(D_2)$.  However this is impossible because all explicit critical level curves of $h$ in $cl(D_2)$ are in $\mathbb{A}$, and $\Lambda_N$ is $\prec$-maximal in $G$.  Therefore the collision point must come in some face of $G\setminus cl(D_2)$.  That is, there is some face $F$ of $G\setminus cl(D_2)$, and some $\widetilde{\epsilon}\in(\epsilon_N,\epsilon)$ such that for some $i\in\{1,\ldots,N-1\}$, $\widetilde{\Lambda_i}$ and $\widetilde{\Lambda_N}$ both intersect $F$.  Since $F\subset cl(D_2)^c$, this means that there are some extension paths $\gamma_i$ and $\gamma_N$ contained in $\widetilde{\Lambda_i}$ and $\widetilde{\Lambda_N}$ respectively which are contained in $F$.

Again since $\Lambda_i$ and $\Lambda_N$ are $\prec$-maximal in $G$, there is no member of $\mathbb{A}$ in $G$ which contains both $\widetilde{\Lambda_i}$ and $\widetilde{\Lambda_N}$ in its bounded faces, and thus $\gamma_i$ and $\gamma_N$ are not separated in $G$ by any member of $\mathbb{A}$.  Therefore $\gamma_i$ and $\gamma_N$ are in the boundary of some common face $F_0$ of $G\setminus\left(cl(D_2)\cup\ds\bigcup_{\Lambda\in\mathbb{A}}\Lambda\right)$.

$F_0$ contains no zeros (explicit or implit) of $h$, so the net change of $\arg(h)$ along $\partial F_0$ is $0$.  Since $\arg(h)$ is constant on any gradient line segment of $\partial F_0$, it follows that the total change in $\arg(h)$ along the bottom edges of $\partial F_0$ is equal in magnitude to the total change in $\arg(h)$ along the top edges of $\partial F_0$.  Select some initial point $u_0$ in a gradient line portion of $\partial F_0$, and parameterize $\partial F_0$ by a map $\psi:[0,1]\to\partial F_0$.  Define $\alpha_0$ to be the choice of $\arg(h(u_0))$ in $[0,2\pi)$, and define $\alpha:[0,1]\to\mathbb{R}$ by $\alpha(t)=\arg(h(\psi(t)))$, chosen so that $\alpha$ is continuous.  If we collapse the parameterization so that it ignores gradient line segments of $\partial F_0$ (ie we mod out $\partial F_0$ by gradient lines), then $\alpha$ alternates between strictly increasing and strictly decreasing.  Now, $\partial F_0$ alternates between top and bottom edges (after modding out by gradient lines), and $\partial F_0$ has several distinct bottom edges (ie the $\gamma_i$ and $\gamma_N$), thus we have a continuous function $\alpha$ on $[0,1]$ which alternates finitely many times between strictly increasing and strictly decreasing (strictly increasing twice or more and strictly decreasing twice or more) and $\alpha(0)=\alpha(1)$ (since the net change in $\arg(h)$ around $\partial F_0$ is $0$).  It follows then that there are two distinct intervals $[a_1,b_1]$ and $[a_2,b_2]$ such that $0\leq a_1<b_1<a_2<b_2\leq1$ such that $\alpha$ is increasing on $(a_1,b_1)$ and $(a_2,b_2)$, and $\alpha$ is decreasing on $(b_1,a_2)$, and there are some $t_1\in(a_1,b_1)$ and $t_2\in(a_2,b_2)$ with $\alpha(t_1)=\alpha(t_2)$.  Let $E_1$ denote the full extension containing the point $\psi(t_1)$ and $E_2$ the full extension containing the point $\psi(t_2)$.  We now join $E_1$ and $E_2$ at the points $\psi(t_1)$ and $\psi(t_2)$ respectively.  We include the resulting graph in $\mathbb{A}$.  This point $\psi(t_1)$ (which now equals $\psi(t_2)$) represents one of the implied critical points of $h$ in $cl(D_2)$.  In this case we define $\arg(h(\psi(t_1)))$ to be the value $\alpha(t_1)$.  Note that since $E_1$ and $E_2$ are each analytic level curve type sets, and are each contained in the unbounded face of the other, the graph obtained by joining the two at a single point is again an analytic critical level curve type set.

We have already shown that full extensions of level curves of $h$ in $cl(D_2)$ are analytic level curve type sets.  We will now show that each member of $\mathbb{A}$ forms a member of $P_a$.

%%For each $1\leq i\leq M$, let $F_i$ denote the component of $G\setminus cl(D_2)$ which is adjacent to $L_i$.  Suppose for the moment that each $\partial F_i$ consists only of $K_i^-$, $K_i^+$, and a length of $\partial G$.  Then by extending the boundary of $G$ through the segments of $cl(D_2)$ one can show that $G$ could not contain any other member of $\mathbb{A}$.  Therefore there must be some $F_i$ whose boundary contains something other than just $K_i^-$, $K_i^+$, and a segment of $\partial G$.  This $F_i$ then must either contain a segment of some $\widetilde{\Lambda_j}$ or have a segment of some $\widetilde{\Lambda_j}$ in its boundary for some $1\leq j\leq N-1$.

Let $\Lambda$ be some member of $\mathbb{A}$.  Define $H(\langle\Lambda\rangle_{Pa})$ to be the value $|h|$ takes on $\Lambda$.

We say that any point in $\Lambda\cap cl(D_2)$ at which $h$ takes a positive real value a distinguished point of $\Lambda$.  If $\gamma_u$ is any extension path used to construct $\Lambda$, let $\alpha_0$ denote any choice of $\arg(h(u))$.  Let $n$ denote the number of integer multiples of $2\pi$ in the interval $(\alpha_0,\alpha_0+\Delta_{\arg}(\gamma_u))$.  Then we distinguish $n$ distinct points in $\gamma_u$.  If $\gamma_u$ is joined to some other extension path just choose these distinguished points to be distributed in a way that coincides with the value of $\arg(h)$ at the join point.

For each vertex $v$ of $\Lambda$, we may define $a(v)\colonequals\arg(h(v))$.

If $F$ is a bounded face of $\Lambda$, then we define $z(F)$ to be the number of distinguished points in $\partial F$.  If $F_1,\ldots,F_n$ is an enumeration of the bounded faces of $\Lambda$, we define $Z(\langle\Lambda\rangle_{P_a})\colonequals\ds\sum_{i=1}^nz(F_i)$.

Again let $F$ be one of the bounded faces of $\Lambda$.  The only thing to verify before concluding that $\langle\Lambda\rangle_{P_a}$, with the auxiliary data we have just defined, is a member of $P_a$ is that if $x_1$ and $x_2$ are distinct vertices of $\Lambda$ in $\partial{F}$ such that $a(x_1)\geq{a(x_2)}$, then there is some distinguished point $z\in\partial{F}$ such that $x_1,z,x_2$ is written in increasing order as they appear in $\partial{F}$.  Assume that $x_1$ and $x_2$ are as described.  Let $L$ denote the segment of $\partial F$ with end points $x_1$ and $x_2$ such that $x_1$ is the initial point of $L$ if $L$ is traversed with positive orientation around $\partial F$.  If $L$ is contained entirely in $cl(D_2)$, the desired result follows directly from the fact that $h$ is analytic on $cl(D_2)$ (since $\arg(h)$ will be strictly increasing as $L$ is traversed).  Suppose now that there is some single extension path $\gamma$ which forms a part of $L$, and $L\setminus\gamma\subset cl(D_2)$.  Let $u$ denote the initial point of $\gamma$ as $\gamma$ is traversed with positive orientation.  Suppose that there is no distinguished point in $L\setminus\gamma$.  Let $\Delta_1$ denote the total change in $\arg(h(z))$ as $z$ traverses $L$ from $x_1$ to $u$.  Let $\Delta_2$ denote the total change in $\arg(h(z))$ as $z$ traverses $L$ from $\mathcal{N}(u)$ to $x_2$.  Then $\arg(h(x_2))=\arg(h(x_1))+\Delta_1+\Delta_{\arg}(\gamma)+\Delta_2$.  Since the choice of $\arg(h(x_2))$ in $[0,2\pi)$ is less than the choice of $\arg(h(x_1))$ in $[0,2\pi)$, it follows that there is at least one multiple of $2\pi$ in the interval $(\arg(h(x_1)),\arg(h(x_1))+\Delta_1+\Delta_{\arg}(\gamma)+\Delta_2)$.  If this multiple occurs in $(\arg(h(x_1)),\arg(h(x_1))+\Delta_1)$ or in $(\arg(h(x_1))+\Delta_1+\Delta_{\arg}(\gamma),\arg(h(x_1))+\Delta_1+\Delta_{\arg}(\gamma)+\Delta_2)$, then there would be a point in $L\setminus \gamma$ at which $h$ takes a positive real value, and thus $L$ has a distinguished point at that point.  If it occurs in the interval $(\arg(h(x_1))+\Delta_1,\arg(h(x_1))+\Delta_1+\Delta_{\arg}(\gamma))$, then by the definition of distinguished points above, $\gamma$ contains a distinguished point, which establishes the desired result.  (If $L$ contains several extension paths, make the appropriate minor changes.)  Therefore we conclude that the full extension of any level curve of $h$ in $cl(D_2)$ gives rise naturally to a member of $P_a$.

We now wish to show that the members of $\mathbb{A}$ together form a member of $PC_a$.  We do this recursively.  Of course the single point members of $\mathbb{A}$ are members of $PC_a$ on their own.  Let $\Lambda$ be some non-single point member of $\mathbb{A}$, and let $G$ denote one of the bounded faces of $\Lambda$.  By our construction of the implied critical points of $h$, there is a unique $\prec$-maximal member $\Lambda_G$ of $\mathbb{A}$ in $G$.  Assume recursively that we have formed a member $\langle\Lambda_G\rangle_{PC_a}$ of $PC_a$ from $\Lambda_G$.  If $\Lambda$ is a simple closed path (so $G$ is the only face of $\Lambda$), we just define $\langle\Lambda\rangle_{PC_a}$ to be $\langle\Lambda_G\rangle_{PC_a}$.  Otherwise, we assign $\langle\Lambda_G\rangle_{PC_a}$ to $G$, and the only remaining thing left to do is determine gradient maps from the distinguished points in $\partial G$ to the distinguished points in $\Lambda_G$.  This may easily be done by just selecting a map which harmonizes with any actual gradient lines of $h$ in $cl(D_2)$ which connect $\partial G$ to $\Lambda_G$.

We conclude finally that the members of $\mathbb{A}$ form a member of $PC_a$.  Let $\langle\Lambda_{cl(D_2)}\rangle_{PC_a}$ denote this member.  By Theorem~\ref{thm: Polynomial critical level curve config existence.}, we may find some polynomial $p\in\mathbb{C}[z]$ with critical level curve configuration equal to $\langle\Lambda_{cl(D_2)}\rangle_{PC_a}$.  Define $G_p\colonequals\{z\in\mathbb{C}:|p(z)|<1\}$.  The furthest outside full extension of a level curve of $h$ in $cl(D_2)$ (ie the one on which $|h|$ is maximal) is a simple closed path.  Let this full extension be denoted $\Lambda_{whole}$.  Let $D_3$ denote the bounded face of $\Lambda_{whole}$.  Remove the members of $\mathbb{A}$ from $D_3$.  Let $F$ denote one of the components of the remaining set (ie one of the components of $D_2\setminus\mathbb{A}$), and let $\widehat{F}$ denote one of the components of $F\cap cl(D_2)$.  Let $\widehat{F}_p$ denote the corresponding portion of $G_p$.  It is not hard to show (by, for example, an extension of the corresponding argument in~\cite{Ri}) that there is an bijective analytic map $\phi:\widehat{F}\to\widehat{F}_p$ such that $h=p\circ\phi$ on $\widehat{F}$.  This is because the image of $\widehat{F}$ under $h$ (or $\widehat{F}_p$ under $p$) is a polar rectangle.

Having defined this on each component of the set remaining from $cl(D_2)$ after the members of $\mathbb{A}$ have been removed, it may easily be seen that the map $\phi$ extends continuously to the points in $cl(D_2)$ which are in members of $\mathbb{A}$ by choice of the polynomial $p$.  By the Schwartz reflection principle this now gives us the desired result.
\end{proof}

%% file: _section4_futurework.tex
\section{DIRECTIONS FOR FUTURE WORK}\label{sect: Future work.}%
As mentioned in Section~\ref{sect: Construction of PC_a.}, the original construction of $PC_a$ (in~\cite{Ri}) parameterized also the possible critical level curve configurations for function pairs $(f,G)$, where $f$ is allowed to be meromorphic in $G$ (which we then call a \textit{generalized finite Blaschke ratio}, since $f$ would pull back to a ratio of finite Blaschke products on the disk), and the result of Theorem~\ref{thm: Conformal equivalence same as topological equivalence.} applies to generalized finite Blaschke ratios.  However there is at present no meromorphic version of Theorem~\ref{thm: Polynomial critical level curve config existence.}.  It would be desirable to develop such a generalization of Theorem~\ref{thm: Polynomial critical level curve config existence.}, and apply the methods found in this paper to obtain the following simply connected version of Theorem~\ref{thm: Analytic implies equiv to poly.} for meromorphic functions.

\begin{customthm}{\ref{thm: Analytic implies equiv to poly.}b}
If $D\subset\mathbb{C}$ is a Jordan domain, and $f$ is meromorphic on $cl(D)$, then there is an injective analytic map $\phi:D\to\mathbb{C}$, and a rational function $r\in\mathbb{C}(z)$ such that $f\equiv r\circ\phi$ on $D$.
\end{customthm}

Finally, it appears that the level curve approach to the proof of Theorem~\ref{thm: Analytic implies equiv to poly.} has some hope of finding the rational conformal model of the function $f$ which has the lowest degree of any rational conformal model.  The way in which one would proceed would be to show that the extended critical level curve configuration which we constructed in Section~\ref{sect: Proof of main result.} is in some sense canonical, and that no critical level curve configuration of a rational function with smaller degree can extend the configuration of level curves of $f$ found in $D$.

%% file: conformalequivalenceoncompactsets.bbl
\begin{thebibliography}{1}

\bibitem{EKS}
P.~Ebenfelt, D.~Khavinson, and H.~S. Shapiro.
\newblock Two-dimensional shapes and lemniscates.
\newblock {\em Contemp. Math.}, 553:45--59, 2011.

\bibitem{Hi}
D.~Hilbert.
\newblock {\"U}ber die entwicklung einer beliebigen analytischen funktion einer
  variabeln in eine uendiche nach ganzen rationalen functionen fortschreitende
  reihe.
\newblock {\em G{\"o}ttinger Nachrichten.}, pages 63--70, 1897.

\bibitem{LS}
G.~Lowther and D.~Speyer.
\newblock Conjecture: Every analytic function on the closed disk is conformally
  a polynomial.
\newblock {\em http://math.stackexchange.com}, jul 2013.
\newblock Accessed: 6-22-15.

\bibitem{P}
A.~Pfluger.
\newblock Ueber die konstruktion riemannscher fl\"achen durch verheftung.
\newblock {\em J. Indian math. Soc.}, 24:401--412, 1961.

\bibitem{Ri}
T.~Richards.
\newblock Level curve configurations and conformal equivalence of meromorphic
  functions.
\newblock {\em Computational Methods and Function Theory}, pages 1--49, 2015.

\bibitem{RY}
T.~Richards and M.~Younsi.
\newblock Conformal models and fingerprints of pseudo-lemniscates.
\newblock {\em submitted to Constructive Approximation}, 2015.

\bibitem{Yo}
M.~Younsi.
\newblock Shapes, fingerprints and rational lemniscates.
\newblock {\em Proc. Amer. Math. Soc.}, to appear.

\end{thebibliography}
